\documentclass[reqno]{amsart}
\usepackage{amsmath}
\begin{document}
\title[limiting behavior of general Euler system]
{Limiting behavior of scaled general Euler equations of compressible fluid flow}

\author[ Sahoo and Sen]
{ Manas R. Sahoo and Abhrojyoti Sen}

\address{ Manas R. Sahoo and Abhrojyoti Sen \newline
School of Mathematical Sciences\\
National Institute of Science Education and Research, Bhubaneswar, 752050, India\\
Homi Bhaba National Institute (HBNI), Training School Complex\\
Anushakti Nagar, Mumbai, 400094, India.}
\email{manas@niser.ac.in, abhroyoti.sen@niser.ac.in}

\thanks{Submitted}

\begin{abstract}
 The aim of this article is to study the limiting behavior of the solutions for the scaled generalized Euler equations of compressible fluid flow.
When the initial data is of Riemann type, we showed the existence of solution which consists of shock waves and rarefaction waves and that the distributional limit of the solutions for this system converges to the solution of a non-strictly hyperbolic system, called one dimensional model for large scale structure formation of universe as the scaling parameter vanishes. An explicit entropy and entropy flux pair is also constructed for the particular flux function (Brio system) and it is shown that the solution constructed is entropy admissible. This is a continuation of our work\cite{sahoo}.
\end{abstract}

\maketitle
\numberwithin{equation}{section}
\numberwithin{equation}{section}
\newtheorem{theorem}{Theorem}[section]
\newtheorem{remark}[theorem]{Remark}
\newtheorem{lem}[theorem]{Lemma}
\newtheorem{definition}[theorem]{Definition}

\section{Introduction} 
 General Euler equations of compressible fluid flow reads
\begin{equation}
\begin{cases}
\begin{aligned}
&u_t +(\frac{u^2}{2}+f(\rho))_x =0,\\
&\rho_t +(u \rho +g(\rho))_x=0.
\end{aligned}
\label{g1.1}
\end{cases}
\end{equation}
We take the initial conditions
\begin{equation}
u(x,0)=u_{0}(x),\,\, \rho(x,0)=\rho_0 (x).
\label{e1.2}
\end{equation}
For the above system, the assumptions on $f(\rho)$ and $g(\rho)$ are the following:
\begin{itemize}
\item[H1.] $f$, $g$ $\in$ $C^3[0,\infty)$ and $f_1=\frac{f^{\prime}}{\rho}\in C^2[0,\infty)$, $g_1=\frac{g^{\prime}}{\rho}\in C^2[0,\infty)$.
\item[H2.] $f_1\geq d$ and $2f^{\prime}_1+g^{\prime}_1(r_1+g_1)\geq 0$, $2f^{\prime}_1+g^{\prime}_1(r_1-g_1)\geq 0$, where $d$ is a fixed positive constant and $r_1=\sqrt{g^2_1+4f_1}$.
\end{itemize}
Under these assumptions H1-H2, the system \eqref{g1.1} is strictly hyperbolic and genuinely nonlinear in both of its characteristic fields\cite{Lu}. Here we are interested in the system \eqref{g1.1} with the following conditions on $f$ and $g$.
\begin{itemize}
\item[A1.] $f\in C^2(0,\infty)$, $f^{\prime}>0$ and $f^{\prime\prime}>0$.
\item[A2.] $g\in C^1(0,\infty)$ and $g$ is any linear decreasing function.
\end{itemize}
It can be easily observed that our assumptions on $f$ and $g$ are compatible with H1 and H2. Since our $g$ is any linear decreasing function, it is enough to work with $g(\rho)=-\rho$. So the system \eqref{g1.1} can be expressed as: 
\begin{equation}
\begin{cases}
\begin{aligned}
&u_t +(\frac{u^2}{2}+f(\rho))_x =0,\\
&\rho_t +(u \rho -\rho)_x=0.
\end{aligned}
\label{g1.2}
\end{cases}
\end{equation}
If $f(\rho)=\frac{\rho^2}{2}$, we get the following Brio system.
\begin{equation}
\begin{cases}
\begin{aligned}
&u_t +(\frac{u^2+\rho^2}{2})_x =0,\\
&\rho_t +(\rho (u-1))_x=0.
\end{aligned}
\label{e1.3}
\end{cases}
\end{equation}
Therefore the system\eqref{g1.2} can be regarded as a generalization of the physically significant system known as Brio system\eqref{e1.3}. The  Brio system \eqref{e1.3} was first derived by M. Brio \cite{brio} and mainly arises as a simplified model in  ideal magnetohydrodynamics(MHD). The study of MHD is based on the idea that the currents in the magnetic fields are inherent from moving electrically conducting fluids.  In this system $(u, \rho)$ represents the velocity of the fluid whose dynamics is determined by magnetohydrodynamics forces. In \cite{lefloch}, equation \eqref{e1.3} was compared with a system whose first equation avoids the non linear term $\frac{1}{2}\rho^2$, such as
\begin{equation}
\begin{cases}
\begin{aligned}
&u_t +(\frac{u^2}{2})_x =0,\\
&\rho_t +(\rho (u-1))_x=0.
\end{aligned}
\label{le1.3}
\end{cases}
\end{equation}
It was shown in \cite{lefloch} that the solution for Riemann problem of the system \eqref{le1.3} contains $\delta$-waves. In \cite{hk1},  $\delta$-shocks are observed in the solution of \eqref{e1.3} by a complex-valued generalization of weak asymptotic method \cite{hk3,danilov} and in \cite{cor} similar result is obtained via a distributional product.  Although uniqueness was an unresolved issue for both of them. Recently the  question of uniqueness is also settled in \cite{hk2} by introducing some nonlinear change of variable in the flux function of the first equation of \eqref{e1.3}.\\
In our present work we are interested in the limiting behavior of the solutions for the scaled version  of \eqref{g1.2} as the scaling parameter approaches zero.  
The scaled version of the system\eqref{g1.2} can be written as
\begin{equation}
\begin{cases}
\begin{aligned}
&u_t +(\frac{u^2}{2}+\epsilon f(\rho))_x =0,\\
&\rho_t +(u \rho -\epsilon \rho)_x=0,
\end{aligned}
\label{sc1.6}
\end{cases}
\end{equation}
where $\epsilon>0$ is introduced as a scaling parameter. Recently \cite{sahoo} deals with the system 
\begin{equation}
\begin{cases}
\begin{aligned}
&u_t +(\frac{u^2}{2}+\epsilon f(\rho))_x =0,\\
&\rho_t +(\rho u)_x=0.
\end{aligned}
\label{pre1.6}
\end{cases}
\end{equation}
One can see that the system \eqref{pre1.6} can be obtained by taking $g(\rho)=0$ and introducing the scaling parameter $\epsilon$ in the system \eqref{g1.1}.
It can be readily seen  that as $\epsilon\rightarrow 0$, formally the above systems \eqref{sc1.6} and \eqref{pre1.6} becomes  
\begin{equation}
\begin{cases}
\begin{aligned}
u_t +(\frac{u^2}{2})_x &=0,\,\,\, x\in \mathbb{R}, t>0\\
\rho_t +(\rho u)_x&=0,\,\,\, x\in \mathbb{R}.
\end{aligned}
\end{cases}
\label{e1.1}
\end{equation}
In \cite{sahoo}, it is shown that the solution of the system \eqref{pre1.6} converges to the solution of \eqref{e1.1} in the sense of distribution. As a continuation of \cite{sahoo},  here our goal is to obtain  the solution of \eqref{e1.1} as a  distributional limit of the solution of \eqref{sc1.6}.\\
The above equation\eqref{e1.1} is a one dimensional model for the large scale structure formation of universe\cite{z1}. This is an example of a non-strictly hyperbolic system, which was studied by many authors \cite{ j2, ma1, ma2, o1, dx, tan}, started with the work of Korchinski\cite{k1}. We study the existence of solution for the equation \eqref{sc1.6} for Riemann type initial data, namely,  
\begin{equation}
\begin{pmatrix}
         u_0 (x)   \\
            \rho_0 (x) \\
         \end{pmatrix}
   =
\begin{cases}
\begin{pmatrix}
         u_l  \\
            \rho_l \\
         \end{pmatrix},\,\,\,\,\,\,\, \textnormal{if}\,\,\,\,\,\,\  x<0\\
 \begin{pmatrix}
         u_r   \\
            \rho_r \\
         \end{pmatrix},\,\,\,\,\,\,\, \textnormal{if}\,\,\,\,\,\,\  x>0.
\end{cases}
\label{e1.4}
\end{equation}
Note that for $\epsilon>0$, the system \eqref{sc1.6} is strictly hyperbolic and both the characteristics fields are genuinely nonlinear (see section $2$). For a strictly hyperbolic system whose characteristics field are either genuinely nonlinear or linearly degenerate, the Lax theory\cite{ba1, da1,lax} can be applied to show the existence of solution for close-by Riemann type initial data. But for our system\eqref{sc1.6}, we show that the existence of solution does not depend on the closeness of initial data. Summarizing the above paragraphs, the main result can be stated as following.
\begin{theorem}
The admissible solution of the system \eqref{sc1.6} with Riemann type initial data \eqref{e1.4} converges to the solution of the non strictly hyperbolic model \eqref{e1.1} in the sense of distribution when the parameter $\epsilon$ goes to zero.
\end{theorem}
We propose a different regularization for the system \eqref{e1.1} by introducing the parameter $\epsilon$ in the flux function of \eqref{g1.1}. Introduction of the scaling parameter $\epsilon>0$ is motivated as follows: The flux $(\frac{u^2}{2}+ \epsilon f(\rho), \rho u-\epsilon \rho)$  in \eqref{sc1.6} compared to the flux $(\frac{u^2}{2}, \rho u)$  in system \eqref{e1.1} gives a more regularized effect. Besides this in presence of $\epsilon>0$ there is no concentration in the solution, however in the absence of $\epsilon$, the system \eqref{sc1.6} becomes \eqref{e1.1} and concentration can occur in the solution which makes it highly singular.

 In this paper first we find solution for the system\eqref{sc1.6} for \textit{any Riemann type initial data} and the solution is a combination of shock and rarefaction waves. Then we study the limiting behavior of these solutions as the parameter $\epsilon$ approaches to zero. We show that the limit is a solution for \eqref{e1.1} which is also vanishing viscosity limit \cite{j2}. 
This type of singular flux function limit approach is very useful for certain systems and can be viewed as an alternative approach of vanishing viscosity to construct solution (which may be singular in nature) for non-strictly hyperbolic systems. In this regard, we refer \cite{Meina} for LeRoux system and \cite{Li1,chen,n1} for isentropic and nonisentropic system of gas dynamics. On a slightly different note, one can see \cite{Meina2} where Riemann solution for \eqref{e1.1} is obtained via a linear approximations of flux function.\\
The  paper is organized as follows. In section $2$, shock and rarefaction curves are described for the system\eqref{sc1.6}. In section $3$, shock-wave solution is constructed  for \eqref{sc1.6}-\eqref{e1.4},  when $u_l > u_r$  and the distributional limit is obtained when the parameter $\epsilon$ approaches to zero and it is shown that limit satisfies \eqref{e1.1} in the sense of the definition\eqref{defsol}. In section $4$, an entropy-entropy flux pair is found for \eqref{e1.3} which satisfies entropy condition for small $\epsilon$. In section $5$, the solution for the case $u_l \leq u_r$  is obtained as a combination of other elementary waves. Lastly in section $6$, we discuss the case when $f(\rho)=\frac{\rho^2}{2}$ and $g(\rho)=-\rho^2$. Also further possibilities are discussed for some general $f$ and $g$.
\section{The Riemann problem}
The co-efficient matrix $A(u,\rho)$  of  the equation \eqref{sc1.6}  is  given by 
\begin{equation*}
A(u,\rho)=\quad
\begin{pmatrix}
u & {\epsilon}f^{\prime}(\rho) \\
\rho & u-\epsilon
\end{pmatrix}.
\quad  
\end{equation*}
Eigenvalues for this co-efficient matrix are the following:
$\lambda_1(u,\rho)=u-\frac{\epsilon}{2}-\frac{1}{2}\sqrt{4\epsilon \rho f^{\prime}(\rho)+\epsilon^2}$ and $\lambda_2(u,\rho)=u-\frac{\epsilon}{2}+\frac{1}{2}\sqrt{4\epsilon \rho f^{\prime}(\rho)+\epsilon^2}$  and the eigenvectors corresponding to $\lambda_1$ and $\lambda_2$ are $X_1=(\frac{\frac{\epsilon}{2}-\frac{1}{2}\sqrt{4\epsilon \rho f^{\prime}(\rho)+\epsilon^2}}{\rho},1)$ and $X_2=(\frac{\frac{\epsilon}{2}+\frac{1}{2}\sqrt{4\epsilon \rho f^{\prime}(\rho)+\epsilon^2}}{\rho},1)$ respectively. Now consider,
\begin{equation*}
\begin{aligned}
\nabla\lambda_1.X_1=\frac{\frac{\epsilon}{2}-\frac{1}{2}\sqrt{4\epsilon \rho f^{\prime}(\rho)+\epsilon^2}}{\rho}-\frac{\epsilon(\rho f^{\prime\prime}(\rho)+f^{\prime}(\rho))}{\sqrt{4\epsilon \rho f^{\prime}(\rho)+\epsilon^2}}
\end{aligned}
\end{equation*}
As $f(\rho)$ and $f^{\prime}(\rho)$ are increasing, we have $\nabla\lambda_1.X_1<0$. A similar calculation shows that $\nabla\lambda_2.X_2>0.$
So each characteristic field is genuinely nonlinear for problem \eqref{sc1.6}. \\
\textbf{Shock curves}: The shock curves $s_1$,$s_2$ through $(u_l,\rho_l)$ and $(u_r,\rho_r)$ are derived from the Rankine-Hugoniot conditions
\begin{equation}
\begin{aligned}
\lambda(u_l-u_r)=&(\frac{1}{2}u_l^2+\epsilon f(\rho_l))-(\frac{1}{2}u_r^2+\epsilon f(\rho_r)),\\
\lambda(\rho_l-\rho_r)=&(\rho_l u_l-\epsilon)-(\rho_r u_r-\epsilon).
\label{2.1}
\end{aligned}
\end{equation}
Eliminating $\lambda$ from \eqref{2.1} and simplifying further, one can get the following quadratic equation
\begin{equation}
(u_l-u_r)^2+\Big(\frac{2\epsilon (\rho_r-\rho_l)}{\rho_l+\rho_r}\Big)(u_l-u_r)-2\epsilon\frac{(\rho_l-\rho_r)(f(\rho_l)-f(\rho_r))}{\rho_l+\rho_r}=0
\label{q2.2}
\end{equation}
Solving the above equation \eqref{q2.2},the admissible part of the shock curves passing through $(u_l, \rho_l)$ are computed as
\begin{equation*}
s_1=\big\{(u,\rho):(u-u_l)=\frac{\rho-\rho_l}{\rho+\rho_l}\Big[\epsilon-\sqrt{\epsilon^2+\frac{2\epsilon(\rho+\rho_l)(f(\rho)-f(\rho_l))}{(\rho-\rho_l)}}\Big],\,\,\, \rho > \rho_l;\,\,\,\,u<u_l\big\},
\end{equation*}
\begin{equation*}
s_2=\big\{(u,\rho):(u-u_l)=\frac{\rho-\rho_l}{\rho+\rho_l}\Big[\epsilon+\sqrt{\epsilon^2+\frac{2\epsilon(\rho+\rho_l)(f(\rho)-f(\rho_l))}{(\rho-\rho_l)}}\Big],\,\,\, \rho < \rho_l;\,\,\,\,u<u_l\big\}.
\end{equation*}\\
\textbf{Rarefaction curves}: The Rarefaction curves $R_1$, $R_2$ passing through $(u_l,\rho_l)$ are the following :\\
\textit{1- Rarefaction curve}: The first Rarefaction curve passing through $(u_l,\rho_l)$ is derived by solving
\begin{equation*}
\frac{du}{d\rho}=\frac{\epsilon-\sqrt{4\epsilon \rho f^{\prime}(\rho)+\epsilon^2}}{2\rho},\,\,\,\,\,\,\,\,\,\,\, u(\rho_l)=u_l;
\end{equation*}
\begin{equation*}
R_1=\big\{(u,\rho):u-u_l=\int_{\rho_l}^{\rho} \frac{\epsilon-\sqrt{4\epsilon \xi f^{\prime}(\xi)+\epsilon^2}}{2\xi},\,\,\, \rho < \rho_l\big\}.
\end{equation*}

\textit{2- Rarefaction curve}: The second Rarefaction curve $R_2$ passing through $(u_l,\rho_l)$ is derived by solving
\begin{equation*}
\frac{du}{d\rho}=\frac{\epsilon+\sqrt{4\epsilon \rho f^{\prime}(\rho)+\epsilon^2}}{2\rho},\,\,\,\,\,\,\,\,\,\,\, u(\rho_l)=u_l;
\end{equation*}
\begin{equation*}
R_2=\big\{(u,\rho):u-u_l=\int_{\rho_l}^{\rho}\frac{\epsilon+\sqrt{4\epsilon \xi f^{\prime}(\xi)+\epsilon^2}}{2\xi} ,\,\,\, \rho > \rho_l\big\}.
\end{equation*}

To solve the equation \eqref{sc1.6} with \eqref{e1.4}, three cases are required to be considered, that is (I)   $u_l>u_r$, (II) $u_l=u_r$ and   (III) $u_l<u_r$. In case (I) for sufficiently small $\epsilon(>0)$, we have solutions as a combination of two shock wave.  For case (II) solutions  are given as the combination of 1-rarefaction and 2-shock curves or 1-shock and 2-rarefaction curves depending upon $\rho_l>\rho_r$ or $\rho_l<\rho_r$ respectively. Finally in case (III) for sufficiently small $\epsilon(>0)$ , the solution consists of two rarefaction waves and vacuum state. We obtain the limit for the solutions in each case and it is exactly matches with the vanishing viscosity limit found in \cite{j2} which satisfies the equation in the sense of definition\eqref{defsol}.

\section{Formation of concentration for $u_l >u_r$}
In this section the limiting behavior for the solution of the equations \eqref{sc1.6}-\eqref{e1.4} for $u_l> u_r$ as $\epsilon$ tends to zero has been studied. We find solution for the system \eqref{sc1.6} satisfying Lax- entropy condition for the case $u_l >u_r$. The first step towards this is to show the existence of the intermediate state. Note that $\rho_l$ and $\rho_r$ are taken positive through out this section.
\begin{theorem}(Existence of an intermediate state).\\
If $u_l > u_r$, there exists an $\eta>0$ such that for any $\epsilon < \eta$, we have a unique intermediate state $(u^ *_{\epsilon}, \rho^ *_{\epsilon})$  which connects $( u_l, \rho_l)$  to  $(u^ *_{\epsilon}, \rho^ *_{\epsilon})$  by 1-shock and  $(u^ *_{\epsilon}, \rho^ *_{\epsilon})$  to $( u_r, \rho_r)$ by 2-shock which satisfies Lax-entropy condition.
\end{theorem}
\begin{proof}
The admissible 1-shock curve passing through $(\bar{u}, \bar{\rho})$ satisfies the following:

\begin{equation}
\begin{aligned}
(u-\bar{u})s_1=&\frac{1}{2}(u^2-\bar{u}^2)+\epsilon(f(\rho)-f(\bar{\rho})),\\
(\rho-\bar{\rho})s_1=&(\rho u-\bar{\rho}\bar{u})+\epsilon(\bar{\rho}-\rho),
 \label{e2.1}
\end{aligned}
\end{equation}
and satisfies the Lax entropy inequality
\begin{equation}
 s_1 < \lambda_1 (\bar{u}, \bar{\rho}), \,\,   \lambda_1 (u, \rho) < s_1 < \lambda_2 (u, \rho).
\label{ee2.2}
\end{equation}

Eliminating $s_1$  from \eqref{e2.1} and simplifying as in Section $2$, we have
\begin{equation}
(u-\bar{u})=\frac{\rho-\bar{\rho}}{\rho+\bar{\rho}}\Big[\epsilon-\sqrt{\epsilon^2+\frac{2\epsilon(\rho+\bar{\rho})(f(\rho)-f(\bar{\rho}))}{(\rho-\bar{\rho})}}\Big].
\label{e2.3}
\end{equation}
We show that for a given $u< \bar{u}$, there exists a unique $\rho > \bar{\rho}$ such that equation \eqref{e2.3} holds. For that let us define a function
\begin{equation}
\begin{aligned}
F(\rho)&=\frac{\rho-\bar{\rho}}{\rho+\bar{\rho}}\Big[\epsilon-\sqrt{\epsilon^2+\frac{2\epsilon(\rho+\bar{\rho})(f(\rho)-f(\bar{\rho}))}{(\rho-\bar{\rho})}}\Big]\\
&=\frac{1-\frac{\bar{\rho}}{\rho}}{1+\frac{\bar{\rho}}{\rho}}\Big[\epsilon-\sqrt{\epsilon^2+\frac{2\epsilon(\rho+\bar{\rho})(\frac{f(\rho)}{\rho}-\frac{f(\bar{\rho})}{\rho})}{(1-\frac{\bar{\rho}}{\rho})}}\Big]
\label{e2.4}
\end{aligned}
\end{equation}
As $\displaystyle{\lim_{\rho\to \infty}}\frac{f(\rho)}{\rho}=\displaystyle{\lim_{\rho\to \infty}} f^\prime (\rho)\leq \infty$, we have$\displaystyle{\lim_{\rho\to \infty}} F(\rho)=-\infty$.
Since $F(\bar{\rho})=0$, we have $F([\bar{\rho},\infty))=(-\infty,0]$.  Hence the equation\eqref{e2.3} is solvable for any given $u\in (-\infty,\bar{u}]$. To prove the uniqueness of $\rho$ in the interval $[\bar{\rho},\infty)$, observe that $F(\rho)$ satisfies the following equation:
\begin{equation*}
F(\rho)^2-\frac{2\epsilon (\rho-\bar{\rho})}{(\rho+\bar{\rho)}}F(\rho)-\frac{2\epsilon (\rho-\bar{\rho})(f(\rho)-f(\bar{\rho}))}{(\rho+\bar{\rho})}=0.
\end{equation*}
Differentiating the above equation, we have
\begin{equation}
F^{\prime}(\rho)\Big[F(\rho)-\frac{2\epsilon (\rho-\bar{\rho})}{(\rho+\bar{\rho})}\Big]=\frac{4\epsilon \bar{\rho}(f(\rho)+f(\bar{\rho}))+2\epsilon(\rho-\bar{\rho})f^{\prime}(\rho)+4\epsilon \bar{\rho}F(\rho)}{(\rho+\bar{\rho})^2}
\end{equation}
Since $F(\rho)-\frac{2\epsilon (\rho-\bar{\rho})}{(\rho+\bar{\rho})}<0$ and $4\epsilon \bar{\rho}(f(\rho)+f(\bar{\rho}))+2\epsilon(\rho-\bar{\rho})f^{\prime}(\rho)+4\epsilon \bar{\rho}F(\rho)>0$, from the above equation we conclude that 
$F(\rho)$ is decreasing in $[\bar{\rho}, \infty)$. This shows the uniqueness of $\rho$.\\
The conditions \eqref{e2.1} and \eqref{ee2.2} hold if and only if $u < \bar{u}$ and $\rho >\bar{\rho}$. In fact, $s_1$ satisfies \eqref{ee2.2} if
\begin{equation}
\begin{aligned}
\frac{\rho u -\bar{\rho}\bar{u}}{\rho -\bar{\rho}}-\epsilon < \bar{u} -\frac{\epsilon}{2}-\frac{1}{2}\sqrt{4\epsilon \bar{\rho}f^{\prime}(\bar{\rho})+\epsilon^2} \,\,\,\, ,\\
u-\frac{\epsilon}{2} -\frac{1}{2}\sqrt{4\epsilon \rho f^{\prime}(\rho)+\epsilon^2} <\frac{\rho u -\bar{\rho}\bar{u}}{\rho -\bar{\rho}}-\epsilon
<u-\frac{\epsilon}{2} +\frac{1}{2}\sqrt{4\epsilon \rho f^{\prime}(\rho)+\epsilon^2}. \\
\label{e2.6}
\end{aligned}
\end{equation}
Now from the first inequality of \eqref{e2.6} one can get,
\begin{equation*}
\frac{\rho(u-\bar{u})}{(\rho-\bar{\rho})}-\frac{\epsilon}{2}<-\frac{1}{2}\sqrt{4\epsilon \bar{\rho}f^{\prime}(\bar{\rho})+\epsilon^2}.
\end{equation*}
Using the equation \eqref{e2.3} the above inequality can be rephrased as
\begin{equation}
\frac{\rho}{\rho+\bar{\rho}}\Big[\epsilon-\sqrt{\epsilon^2+\frac{2\epsilon(\rho+\bar{\rho})(f(\rho)-f(\bar{\rho}))}{(\rho-\bar{\rho})}}\Big]-\frac{\epsilon}{2}<-\frac{1}{2}\sqrt{4\epsilon \bar{\rho}f^{\prime}(\bar{\rho})+\epsilon^2}.
\label{e3.7}
\end{equation}
To prove the above inequality \eqref{e3.7}, we consider
\begin{equation}
\begin{aligned}
G(\rho)&=\frac{\rho}{\rho+\bar{\rho}}\Big[\epsilon-\sqrt{\epsilon^2+\frac{2\epsilon(\rho+\bar{\rho})(f(\rho)-f(\bar{\rho}))}{(\rho-\bar{\rho})}}\Big]-\frac{\epsilon}{2}.
\end{aligned}
\label{ee3.8}
\end{equation}
Now we claim that the above function $G(\rho)$ is decreasing. Assuming that the claim is true, let us complete the proof of the inequality \eqref{e3.7}. Since $G(\rho)$ is decreasing and $\rho>\bar{\rho}$, we have $G(\rho)<G(\bar{\rho})$. Note that, employing mean value theorem on $f(\rho)$, \eqref{ee3.8} can be written as
\begin{equation*}
G(\rho)=\frac{\rho}{\rho+\bar{\rho}}\Big[\epsilon-\sqrt{\epsilon^2+2\epsilon(\rho+\bar{\rho})f^{\prime}(\xi_{\rho})}\Big]-\frac{\epsilon}{2},\,\,\,\,\xi_{\rho}\in [\bar{\rho},\rho].
\end{equation*}
Therefore, $$G(\bar{\rho})=-\frac{1}{2}\sqrt{\epsilon^2+4\epsilon\bar{\rho}f^{\prime}(\xi_{\bar{\rho}})}$$
As $G(\rho)$ is decreasing and $\rho>\bar{\rho}$, we have
\begin{equation*}
G(\rho)<-\frac{1}{2}\sqrt{\epsilon^2+4\epsilon\bar{\rho}f^{\prime}(\xi_{\bar{\rho}})}
\end{equation*}
So it is enough to show that
$$-\frac{1}{2}\sqrt{\epsilon^2+4\epsilon\bar{\rho}f^{\prime}(\xi_{\bar{\rho}})}<-\frac{1}{2}\sqrt{\epsilon^2+4\epsilon \bar{\rho}f^{\prime}(\bar{\rho})} $$
This is evident since $f^{\prime}(\rho)$ is increasing. Now we show that $G(\rho)$ is a decreasing function.
Differentiating \eqref{ee3.8} one can get
\begin{equation}
\begin{aligned}
G^{\prime}(\rho)&=\frac{-\epsilon \rho \frac{d}{d\rho}\Big[\frac{(\rho+\bar{\rho})(f(\rho)-f(\bar{\rho}))}{(\rho-\bar{\rho})}\Big]}{(\rho+\bar{\rho})\sqrt{\epsilon^2+\frac{2\epsilon(\rho+\bar{\rho})(f(\rho)-f(\bar{\rho}))}{(\rho-\bar{\rho})}}}+\frac{\bar{\rho}\Big[\epsilon-\sqrt{\epsilon^2+\frac{2\epsilon(\rho+\bar{\rho})(f(\rho)-f(\bar{\rho}))}{(\rho-\bar{\rho})}}\Big]}{(\rho+\bar{\rho})^2}.
\end{aligned}
\label{dif3.10}
\end{equation}
Now let us analyze the numerator of the first term of \eqref{dif3.10}. Consider,
\begin{equation}
\begin{aligned}
\frac{d}{d\rho}\Big[\frac{(\rho+\bar{\rho})(f(\rho)-f(\bar{\rho}))}{(\rho-\bar{\rho})}\Big]\\
=\frac{(\rho^2-\bar{\rho}^2)f^{\prime}(\rho)-2\bar{\rho}(f(\rho)-f(\bar{\rho}))}{(\rho-\bar{\rho})^2}.
\end{aligned}
\label{dif3.11}
\end{equation}
Since $f^{\prime}(\rho)$ is increasing, a use of mean value theorem on $f(\rho)$ in the interval $[\bar{\rho},\rho]$ shows that $(\rho^2-\bar{\rho}^2)f^{\prime}(\rho)-2\bar{\rho}(f(\rho)-f(\bar{\rho}))>0$. So from \eqref{dif3.10} we conclude that $G^{\prime}(\rho)<0$, i.e, $G(\rho)$ is decreasing. This proves our claim.\\
Now the second inequality of \eqref{e2.6} can be rewritten as
\begin{equation}
\begin{aligned}
&\frac{\bar{\rho}(u-\bar{u})}{(\rho-\bar{\rho})}-\frac{\epsilon}{2} < \frac{1}{2}\sqrt{\epsilon^2+4\epsilon \rho f^{\prime}(\rho)},\\
&\frac{\bar{\rho}(u-\bar{u})}{(\rho-\bar{\rho})}-\frac{\epsilon}{2}>-\frac{1}{2}\sqrt{\epsilon^2+4\epsilon \rho f^{\prime}(\rho)}.
\end{aligned}
\label{inq3.12}
\end{equation}
As $\rho>\bar{\rho}$, $u<\bar{u}$, the first inequality of \eqref{inq3.12} is evident. Again using the equation \eqref{e2.3}, the second inequality of \eqref{inq3.12} can be written as
\begin{equation*}
\frac{\bar{\rho}}{\rho+\bar{\rho}}\Big[\epsilon-\sqrt{\epsilon^2+\frac{2\epsilon(\rho+\bar{\rho})(f(\rho)-f(\bar{\rho}))}{(\rho-\bar{\rho})}}\Big]-\frac{\epsilon}{2}>-\frac{1}{2}\sqrt{\epsilon^2+4\epsilon \rho f^{\prime}(\rho)}.
\end{equation*}
To prove the above inequality, we consider,
\begin{equation}
\begin{aligned}
H(\bar{\rho})&=\frac{\bar{\rho}}{\rho+\bar{\rho}}\Big[\epsilon-\sqrt{\epsilon^2+\frac{2\epsilon(\rho+\bar{\rho})(f(\rho)-f(\bar{\rho}))}{(\rho-\bar{\rho})}}\Big]-\frac{\epsilon}{2}.
\end{aligned}
\end{equation}
In a similar way as above we can show that $H(\bar{\rho})$ is a decreasing function of $\bar{\rho}$ and since $\bar{\rho}<\rho$, we have $H(\bar{\rho})>H(\rho)$. Now following the similar steps as above one gets the second inequality of \eqref{inq3.12}. Note that the above inequality is independent of $\epsilon$ and holds for any $(u,\rho)$ and $(\bar{u},\bar{\rho})$ satisfying the condition $u<\bar{u}$ and $\rho>\bar{\rho}$.\\
Therefore,  the branch of the curve satisfying \eqref{e2.1} and \eqref{ee2.2} can be parameterized by a $C^1$ function $\rho_1:(-\infty, \bar{u}] \rightarrow [\bar{\rho}, \infty)$
with the parameter $u$.

 From the equation \eqref{e2.3}, $\rho_1(u)$ satisfies 

\begin{equation}
(u-\bar{u})= \frac{\rho_1(u)-\bar{\rho}}{\rho_1(u)+\bar{\rho}}\Big[\epsilon-\sqrt{\epsilon^2+\frac{2\epsilon(\rho_1(u)+\bar{\rho})(f(\rho_1(u))-f(\bar{\rho}))}{(\rho_1(u)-\bar{\rho})}}\Big].
\label{e2.10}
\end{equation}
Differentiating the above equation \eqref{e2.10} with respect to the parameter $u$, we have
\begin{equation}
\begin{aligned}
1&=\Big[\frac{\rho_1(u)-\bar{\rho}}{\rho_1(u)+\bar{\rho}}\frac{-\epsilon \frac{d}{d\rho}\Big[\frac{(\rho_1(u)+\bar{\rho})(f(\rho_1(u))-f(\bar{\rho}))}{(\rho_1(u)-\bar{\rho})}\Big]}{\sqrt{\epsilon^2+\frac{2\epsilon(\rho_1(u)+\bar{\rho})(f(\rho_1(u))-f(\bar{\rho}))}{(\rho_1(u)-\bar{\rho})}}}\\
&+\frac{2\bar{\rho}}{(\rho_1(u)+\bar{\rho})^2}\Big(\epsilon-\sqrt{\epsilon^2+\frac{2\epsilon(\rho_1(u)+\bar{\rho})(f(\rho_1(u))-f(\bar{\rho}))}{(\rho_1(u)-\bar{\rho})}}\Big)\Big]\rho^{\prime}_1(u).
\end{aligned}
\label{shock1}
\end{equation}
Since $\rho_1(u)>\bar{\rho}$ and $f^{\prime}(.)$ is increasing, from \eqref{dif3.11} we have
$$\frac{d}{d\rho}\Big[\frac{(\rho_1(u)+\bar{\rho})(f(\rho_1(u))-f(\bar{\rho}))}{(\rho_1(u)-\bar{\rho})}\Big]>0. $$
Now since $\rho_1(u)>\bar{\rho}$, the first term in the right hand side of \eqref{shock1} is negative and the second term is also negative.
Therefore we conclude that $\rho^{\prime}_1(u)<0$.\\
Similarly the branch of the curve satisfying 
\begin{equation*} 
s_1 >\lambda_2 ({u},  {\rho}), \,\,   \lambda_1 (\bar{u}, \bar{\rho}) < s_1 < \lambda_2 (\bar{u}, \bar{\rho}),
\end{equation*} 
is the admissible 2-shock curve which can be parameterized by a $C^1$ function $\rho_2:(-\infty, \bar{u}] \rightarrow (-\infty, \bar{\rho}]$
with the parameter $u$. 

Also $\rho_2$ satisfies the following equation: 
\begin{equation}
(u-\bar{u})= \frac{\rho_2(u)-\bar{\rho}}{\rho_2(u)+\bar{\rho}}\Big[\epsilon+\sqrt{\epsilon^2+\frac{2\epsilon(\rho_2(u)+\bar{\rho})(f(\rho_2(u))-f(\bar{\rho}))}{(\rho_2(u)-\bar{\rho})}}\Big].
\label{e2.9}
\end{equation}
Differentiating the above equation \eqref{e2.9} we have,
\begin{equation}
\begin{aligned}
1&=\Big[\frac{1}{(\rho_2(u)+\bar{\rho})}\Big\{\frac{\epsilon (\rho_2(u)-\bar{\rho})\frac{d}{d\rho}\Big[\frac{(\rho_2(u)+\bar{\rho})(f(\rho_2(u))-f(\bar{\rho}))}{(\rho_2(u)-\bar{\rho})}\Big]}{\sqrt{\epsilon^2+\frac{2\epsilon(\rho_2(u)+\bar{\rho})(f(\rho_2(u))-f(\bar{\rho}))}{(\rho_2(u)-\bar{\rho})}}}\\
&+\Big(\epsilon+\sqrt{\epsilon^2+\frac{2\epsilon(\rho_2(u)+\bar{\rho})(f(\rho_2(u))-f(\bar{\rho}))}{(\rho_2(u)-\bar{\rho})}}\Big)\Big\}\\
&-\frac{(\rho_2(u)-\bar{\rho})\Big[\epsilon+\sqrt{\epsilon^2+\frac{2\epsilon(\rho_2(u)+\bar{\rho})(f(\rho_2(u))-f(\bar{\rho}))}{(\rho_2(u)-\bar{\rho})}}\Big]}{(\rho_2(u)+\bar{\rho})^2}\Big]\rho^{\prime}_2(u).
\end{aligned}
\label{shock2}
\end{equation}
Note that, since $\rho_2(u)<\bar{\rho}$ the second term of the above equation \eqref{shock2} is positive. Now we determine the sign of the first  term. To determine the sign, we calculate
\begin{equation}
\begin{aligned}
&\frac{\epsilon (\rho_2(u)-\bar{\rho})\frac{d}{d\rho}\Big[\frac{(\rho_2(u)+\bar{\rho})(f(\rho_2(u))-f(\bar{\rho}))}{(\rho_2(u)-\bar{\rho})}\Big]}{\sqrt{\epsilon^2+\frac{2\epsilon(\rho_2(u)+\bar{\rho})(f(\rho_2(u))-f(\bar{\rho}))}{(\rho_2(u)-\bar{\rho})}}}
+\Big(\epsilon+\sqrt{\epsilon^2+\frac{2\epsilon(\rho_2(u)+\bar{\rho})(f(\rho_2(u))-f(\bar{\rho}))}{(\rho_2(u)-\bar{\rho})}}\Big)\\
&=\frac{\epsilon (\rho_2(u)-\bar{\rho})\frac{d}{d\rho}\Big[\frac{(\rho_2(u)+\bar{\rho})(f(\rho_2(u))-f(\bar{\rho}))}{(\rho_2(u)-\bar{\rho})}\Big]+\frac{2\epsilon(\rho_2(u)+\bar{\rho})(f(\rho_2(u))-f(\bar{\rho}))}{(\rho_2(u)-\bar{\rho})}}{\sqrt{\epsilon^2+\frac{2\epsilon(\rho_2(u)+\bar{\rho})(f(\rho_2(u))-f(\bar{\rho}))}{(\rho_2(u)-\bar{\rho})}}}\\
&+\frac{\epsilon\Big[\epsilon+\sqrt{\epsilon^2+\frac{2\epsilon(\rho_2(u)+\bar{\rho})(f(\rho_2(u))-f(\bar{\rho}))}{(\rho_2(u)-\bar{\rho})}}\Big] }{\sqrt{\epsilon^2+\frac{2\epsilon(\rho_2(u)+\bar{\rho})(f(\rho_2(u))-f(\bar{\rho}))}{(\rho_2(u)-\bar{\rho})}}}.
\end{aligned}
\label{diff3.16}
\end{equation}
Now observe that, in the view of \eqref{dif3.11} and employing mean value theorem on $f(.)$ in the interval $[\rho_2(u),\bar{\rho}]$, the the numerator of the first term of the above equation, i.e
$$\epsilon (\rho_2(u)-\bar{\rho})\frac{d}{d\rho}\Big[\frac{(\rho_2(u)+\bar{\rho})(f(\rho_2(u))-f(\bar{\rho}))}{(\rho_2(u)-\bar{\rho})}\Big]+\frac{2\epsilon(\rho_2(u)+\bar{\rho})(f(\rho_2(u))-f(\bar{\rho}))}{(\rho_2(u)-\bar{\rho})}>0.$$
To show the above inequality we also used the fact that $f(.)$ is increasing. Therefore from \eqref{shock2},  we conclude that $\rho^{\prime}_2(u)>0$.\\
 Now consider the branch of the curve passing through $(u_r, \rho_r)$ satisfying the condition $u>u_r,\,\,\, \rho > \rho_r$. In a similar way as above it can be parameterized by a $C^1$- curve
$\rho^{*}_2:[u_r, \infty)\to [\rho_r, \infty)$. Then for any given point $(\alpha, \beta)$, the part of the curve $\rho^{*}_2$ connecting $(\alpha, \beta)$ to $(u_r,\rho_r)$ will be the admissible 2-shock curve.
Let us denote the admissible 1-shock curve passing through $(u_l, \rho_l)$ as $\rho^{*}_1 $. From the previous analysis, this is parameterized by a $C^1$ curve  $\rho^{*}_1 :  (- \infty, u_l] \rightarrow [\rho_l, \infty)$. Then
$\rho^{*}_1(u_r)$  satisfies \eqref{e2.10} with $\rho_1(u)$  and ${u}$ replaced by $\rho^{*}_1 (u_r)$ and $u_r$ respectively, and $\bar{u}$, $\bar{\rho}$  replaced by $u_l$ and $\rho_l$ respectively, i.e.,
\begin{equation}
(u_r-u_l)= \frac{\rho^*_1(u_r)-\rho_l}{\rho^*_1(u_r)+\rho_l}\Big[\epsilon-\sqrt{\epsilon^2+\frac{2\epsilon(\rho^*_1(u_r)+\rho_l)(f(\rho^*_1(u_r))-f(\rho_l))}{(\rho^*_1(u_r)-\rho_l)}}\Big].
\label{ee2.10}
\end{equation}
 Again $\rho^{*}_2 (u_l)$  satisfies  \eqref{e2.9} with 
$\rho_2(u)$ and ${u}$ replaced by $\rho^*_2(u_l)$ and $u_l$ respectively, and $\bar{u},\bar{\rho}$  replaced by  $u_r$ and $\rho_r$ respectively, i.e.,
\begin{equation}
(u_l-u_r)= \frac{\rho^*_2(u_l)-\rho_r}{\rho^*_2(u_l)+\rho_r}\Big[\epsilon+\sqrt{\epsilon^2+\frac{2\epsilon(\rho^*_2(u_l)+\rho_r)(f(\rho^*_2(u_l))-f(\rho_r))}{(\rho^*_2(u_l)-\rho_r)}}\Big].
\label{ee2.9}
\end{equation}
 It is evident from \eqref{ee2.10} and \eqref{ee2.9} that $\rho^{*}_1(u_r)$ and  $\rho^{*}_2 (u_l)$  tend to $\infty$ as $\epsilon$ tends to zero. Suppose $\rho^{*}_1(u_r)$ and $\rho^{*}_2 (u_l)$ are finite up to a subsequence as $\epsilon$ tends to zero, then \eqref{ee2.10} and \eqref{ee2.9} implies $u_l=u_r$, which is not the case.   Therefore there exists an $\eta >0$  such that for any $\epsilon< \eta, $  one has  $ \rho^{*}_2 (u_l)>\rho_l$ and  $ \rho^{*}_1 (u_r)>\rho_r$. Now let us consider the function $\rho^{*}_1- \rho^{*}_2:[u_r, u_l]\to \mathbb{R}$. Since $ \rho^{*}_1(u_l)-\rho^{*}_2(u_l)=\rho_l -\rho^{*}_2 (u_l) <0$  and  $ \rho^{*}_1(u_r)-\rho^{*}_2 (u_r)= \rho^{*}_1(u_r)-\rho_r >0$,  by intermediate value theorem there exists a point $ u^{*}_{\epsilon}$ such that  $\rho^{*}_1( u^{*}_{\epsilon})=\rho^{*}_2 ( u^{*}_{\epsilon})=\rho^{*}_{\epsilon}$(say). 
The uniqueness of $\rho^{*}_{\epsilon}$ follows from the fact that $\rho^{*}_1$  is strictly decreasing and $\rho^{*}_2$ is strictly increasing. Since we are considering only admissible part of the curves,  Lax entropy condition 
holds. This completes the proof.\\
\end{proof}
The next tusk is to determine the limit of the problem \eqref{sc1.6} for the shock case. First we define $\delta$-distribution and state a Lemma from \cite{sahoo} without proof .\\
\begin{definition}
 A weighted $\delta$-distribution ``$ d(t)\delta_{x=c(t)}$"  concentrated on a smooth curve $x=c(t)$ can be defined by
\begin{equation*}
\langle\, d(t)\delta_{x=c(t)},\varphi(x,t)\rangle=\int_{0}^{\infty}d(t)\varphi(c(t),t)dt
\end{equation*}
for all $\varphi\in C^\infty_c(\mathbb{R}\times (0,\infty))$.
\end{definition}
\begin{lem}
Suppose $a_{\epsilon}(t)(>0)$ and $b_{\epsilon}(t)(>0)$ converge uniformly to $0$ on compact subsets of $(0,\infty)$ as $\epsilon$ tends to zero.  Also assume that $d_{\epsilon}(t)$ converges to $d(t)$  uniformly on compact subsets of $ (0,\infty)$ as $\epsilon$ tends to zero. Then 
\begin{equation*}
\frac{1}{b_{\epsilon}(t)+a_{\epsilon}(t)} d_{\epsilon}(t)
\displaystyle{\chi_{(c(t)-a_{\epsilon}(t), c(t)+b_{\epsilon}(t))}}(x)
\end{equation*}
converges to $d(t)\delta_{x=c(t)}$ in the sense of distribution.
\end{lem}
\begin{proof}
see\cite{sahoo}
\end{proof}
\begin{theorem}(Limiting behavior as $\epsilon\to0$ )\\
The point wise limit of $u^{\epsilon}$ is $u$ and is given by 
\begin{equation*}
u(x,t)=\begin{cases}
u_l,\,\,\, \textnormal{if}\,\, x< \frac{u_l +u_r}{2}t\\
\frac{u_l+u_r}{2},\,\,\,  \textnormal {if}\,\,\,   x=\frac{u_l+u_r}{2}t\\
u_r,\,\,\, \textnormal{if}\,\,\, x> \frac{u_l +u_r}{2}t.
\end{cases}
\end{equation*}
The distributional limit of $\rho^{\epsilon}$ is $\rho$ and is given by
\begin{equation*}
\rho(x,t)=\begin{cases}
\rho_l,\,\,\, \textnormal{if}\,\,\,\, x< \frac{u_l +u_r}{2}t\\
(u_l-u_r)\frac{\rho_l +\rho_r}{2}t \delta_{x=\frac{u_l +u_r}{2}t},\,\,\, \textnormal{if}\,\,\,\, x= \frac{u_l +u_r}{2}t\\
\rho_r,\,\,\, \textnormal{if}\,\,\,\, x> \frac{u_l +u_r}{2}t.
\end{cases}
\end{equation*}
\label{theorem3.3}
\end{theorem}
\begin{proof}
From the previous Theorem $3.1$, we have $(u^{*}_{\epsilon}, \rho^{*}_{\epsilon})$ satisfies the following equations:
\begin{equation}
\begin{aligned}
(u^*_{\epsilon}-u_l)=\frac{\rho^*_{\epsilon}-\rho_l}{\rho^*_{\epsilon}+\rho_l}\Big[\epsilon-\sqrt{\epsilon^2+\frac{2\epsilon(\rho^*_{\epsilon}+\rho_l)(f(\rho^*_{\epsilon})-f(\rho_l))}{(\rho^*_{\epsilon}-\rho_l)}}\Big],\\
(u^*_{\epsilon}-u_r)=\frac{\rho^*_{\epsilon}-\rho_r}{\rho^*_{\epsilon}+\rho_r}\Big[\epsilon+\sqrt{\epsilon^2+\frac{2\epsilon(\rho^*_{\epsilon}+\rho_r)(f(\rho^*_{\epsilon})-f(\rho_r))}{(\rho^*_{\epsilon}-\rho_r)}}\Big].
\end{aligned}
\label{e2.11}
\end{equation}
We know $u^{*}_{\epsilon} \in (u_r, u_l)$.  So the sequence $u^{*}_{\epsilon}$ is bounded. We claim that \textit{$\rho^{*}_{\epsilon}$ is unbounded as $\epsilon$ tends to zero}. In fact,  if $\rho^*_{\epsilon}$ is bounded, then it has a convergent subsequence still denoted by $\rho^*_{\epsilon}$ and it converges to $\rho^*(\neq \rho_l , \rho_r)$ as $\epsilon$ tends to zero. Then passing to the limit as $\epsilon\to 0$ in the above equation \eqref{e2.11}, we have $u^*=u_l=u_r$. Now suppose $\rho^*=\rho_l$, since $\displaystyle{\lim_{\rho^*_{\epsilon}\to\rho_l}}\frac{f(\rho^*_{\epsilon})-f(\rho_l)}{\rho^*_{\epsilon}-\rho_l}>0$, we have $u^*_{\epsilon}=u_l=u_r$. Similar argument works when $\rho^*=\rho_r$. In all of the cases we get a contradiction.\\
 So for subsequence of $u^{*}_{\epsilon}$ and  $\rho^{*}_{\epsilon}$  still denoted  as $u^{*}_{\epsilon}$  and $\rho^{*}_{\epsilon}$ respectively  
we have that $u^{*}_{\epsilon}$ converges to $u^{*}$  and  $\rho^{*}_{\epsilon}$ tend to $+\infty$ as $\epsilon\to 0$. Passing to the limit  for this subsequence in \eqref{e2.11}, we get
\begin{equation*}
\begin{aligned}
 (u^{*}-u_l)&= -\sqrt{l}\\
 (u^{*}-u_r)&= \sqrt{l},
\end{aligned}
\end{equation*}
where $\displaystyle {\lim_{\epsilon \rightarrow 0}}\,\,\, 2\epsilon\big(f(\rho^*_{\epsilon})-f(\rho_l)\big)=\displaystyle {\lim_{\epsilon \rightarrow 0}}\,\,\, 2\epsilon\big(f(\rho^*_{\epsilon})-f(\rho_r)\big)=l.$
Solving the above two equations one can easily find
\begin{equation}
 u^{*}=\frac{u_l +u_r}{2}\,\, \textnormal{and}\,\, l=\frac{1}{4}(u_l -u_r)^2.
\label{e2.15}
\end{equation}
Now from the above Theorem $3.1$, we see that the  intermediate state $(u^*_{\epsilon}, \rho^*_{\epsilon})$ satisfies the equation \eqref{e2.1}. That is,
\begin{equation}
\begin{aligned}
(u-u^*_{\epsilon})s_{1,\epsilon}=&\frac{1}{2}(u^2-u^{*2}_{\epsilon})+\epsilon(f(\rho)-f(\rho^*_{\epsilon})),\\
(\rho-\rho^{*}_{\epsilon})s_{1,\epsilon}=&(\rho u-\rho^*_{\epsilon}u^*_{\epsilon})-\epsilon(\rho-\rho^*_{\epsilon}),
 \label{e3.16}
\end{aligned}
\end{equation}
where $s_{1,\epsilon}$ is the 1-shock speed.
From the above equation we have,
\begin{equation}
s_{1,\epsilon}=\frac{\rho u-\rho^*_{\epsilon}u^*_{\epsilon}}{\rho-\rho^{*}_{\epsilon}}-\epsilon.
\end{equation}
Now we observe that, using the first equation of \eqref{e2.11}(with $u_l$ replaced by $u$) $s^*_1$ can be rewritten as 
\begin{equation}
\begin{aligned}
s_{1,\epsilon}=&\frac{\rho u-\rho^*_{\epsilon}u^*_{\epsilon}}{\rho-\rho^{*}_{\epsilon}}-\epsilon\\
=&\frac{(u+u^*_{\epsilon})(\rho-\rho^*_{\epsilon})+(u-u^*_{\epsilon})(\rho+\rho^*_{\epsilon})}{2(\rho-\rho^*_{\epsilon})}-\epsilon\\
=&\frac{u+u^*_{\epsilon}}{2}-\frac{\epsilon}{2}-\frac{1}{2}\Big[\epsilon-\frac{(u-u^*_{\epsilon})(\rho+\rho^*_{\epsilon})}{\rho-\rho^*_{\epsilon}}\Big]\\
=&\frac{u+u^*_{\epsilon}}{2}-\frac{\epsilon}{2}-\frac{1}{2}\sqrt{\epsilon^2+\frac{2\epsilon(\rho+\rho^*_{\epsilon})(f(\rho)-f(\rho^*_{\epsilon}))}{(\rho-\rho^*_{\epsilon})}}.
\end{aligned}
\label{shock speed}
\end{equation}
Similarly using the second equation of \eqref{e2.11} $s^*_2$ can be written as
\begin{equation}
s_{2,\epsilon}=\frac{u+u^*_{\epsilon}}{2}-\frac{\epsilon}{2}+\frac{1}{2}\sqrt{\epsilon^2+\frac{2\epsilon(\rho+\rho^*_{\epsilon})(f(\rho)-f(\rho^*_{\epsilon}))}{(\rho-\rho^*_{\epsilon})}}.
\end{equation}
where $s_{2,\epsilon}$ is the 2-shock speed.\\
The solution for $(u^{\epsilon},\rho^{\epsilon })$ is given by
\begin{equation}
 \begin{aligned}
u^{\epsilon}(x,t)=\begin{cases} 
                            u_l  \,\,\,\,\textnormal{ if} \,\,\,\, x<\Big(\frac{u_l+u^*_{\epsilon}}{2}-\frac{\epsilon}{2}-\frac{1}{2}\sqrt{\epsilon^2+\frac{2\epsilon(\rho_l+\rho^*_{\epsilon})(f(\rho_l)-f(\rho^*_{\epsilon}))}{(\rho_l-\rho^*_{\epsilon})}}\Big)t\\
                            u_{\epsilon}^* \,\,\,\, \textnormal{ if}\,\,\,\,  \Big(\frac{u_l+u^*_{\epsilon}}{2}-\frac{\epsilon}{2}-\frac{1}{2}\sqrt{\epsilon^2+\frac{2\epsilon(\rho_l+\rho^*_{\epsilon})(f(\rho_l)-f(\rho^*_{\epsilon}))}{(\rho_l-\rho^*_{\epsilon})}}\Big)t<x\\\,\,\,\,\,\,\,\,\,\,\,\,\,\,\,\, <
\Big(\frac{u_r+u^*_{\epsilon}}{2}-\frac{\epsilon}{2}+\frac{1}{2}\sqrt{\epsilon^2+\frac{2\epsilon(\rho_r+\rho^*_{\epsilon})(f(\rho_r)-f(\rho^*_{\epsilon}))}{(\rho_r-\rho^*_{\epsilon})}}\Big)t\\
                            u_r \,\,\,\, \textnormal{ if} \,\,\,\, x> \Big(\frac{u_r+u^*_{\epsilon}}{2}-\frac{\epsilon}{2}+\frac{1}{2}\sqrt{\epsilon^2+\frac{2\epsilon(\rho_r+\rho^*_{\epsilon})(f(\rho_r)-f(\rho^*_{\epsilon}))}{(\rho_r-\rho^*_{\epsilon})}}\Big)t,
                           \end{cases}
 \end{aligned}
\label{e3.14}
\end{equation}
 \begin{equation}
 \begin{aligned}
\rho^{\epsilon}(x,t)=\begin{cases} 
                            \rho_l  \,\,\,\,\textnormal{ if} \,\,\,\, x<\Big(\frac{u_l+u^*_{\epsilon}}{2}-\frac{\epsilon}{2}-\frac{1}{2}\sqrt{\epsilon^2+\frac{2\epsilon(\rho_l+\rho^*_{\epsilon})(f(\rho_l)-f(\rho^*_{\epsilon}))}{(\rho_l-\rho^*_{\epsilon})}}\Big)t\\
                            \rho_{\epsilon}^* \,\,\,\, \textnormal{ if}\,\,\,\,  \Big(\frac{u_l+u^*_{\epsilon}}{2}-\frac{\epsilon}{2}-\frac{1}{2}\sqrt{\epsilon^2+\frac{2\epsilon(\rho_l+\rho^*_{\epsilon})(f(\rho_l)-f(\rho^*_{\epsilon}))}{(\rho_l-\rho^*_{\epsilon})}}\Big)t<x\\\,\,\,\,\,\,\,\,\,\,\,\,\,\,\,\, <
\Big(\frac{u_r+u^*_{\epsilon}}{2}-\frac{\epsilon}{2}+\frac{1}{2}\sqrt{\epsilon^2+\frac{2\epsilon(\rho_r+\rho^*_{\epsilon})(f(\rho_r)-f(\rho^*_{\epsilon}))}{(\rho_r-\rho^*_{\epsilon})}}\Big)t\\
                            \rho_r \,\,\,\, \textnormal{ if} \,\,\,\, x> \Big(\frac{u_r+u^*_{\epsilon}}{2}-\frac{\epsilon}{2}+\frac{1}{2}\sqrt{\epsilon^2+\frac{2\epsilon(\rho_r+\rho^*_{\epsilon})(f(\rho_r)-f(\rho^*_{\epsilon}))}{(\rho_r-\rho^*_{\epsilon})}}\Big)t,
                           \end{cases}
 \end{aligned}
\label{e3.15}
\end{equation}
As $u^*_{\epsilon}$ converges to $u^*=\frac{u_l+u_r}{2}$ as $\epsilon \rightarrow 0$, we have the limit for $u(x,t)$ as stated in the theorem.\\
From \eqref{e2.15}, one can show that
 $$\displaystyle {\lim_{\epsilon \rightarrow 0}}\Big[\frac{u_l+u^*_{\epsilon}}{2}-\frac{\epsilon}{2}-\frac{1}{2}\sqrt{\epsilon^2+\frac{2\epsilon(\rho_l+\rho^*_{\epsilon})(f(\rho_l)-f(\rho^*_{\epsilon}))}{(\rho_l-\rho^*_{\epsilon})}}\Big]=\frac{u_l +u_r}{2},$$
and $$\displaystyle {\lim_{\epsilon \rightarrow 0}}\Big[\frac{u_r+u^*_{\epsilon}}{2}-\frac{\epsilon}{2}+\frac{1}{2}\sqrt{\epsilon^2+\frac{2\epsilon(\rho_r+\rho^*_{\epsilon})(f(\rho_r)-f(\rho^*_{\epsilon}))}{(\rho_r-\rho^*_{\epsilon})}}\Big]
=\frac{u_l +u_r}{2}.$$

Let us denote

\begin{equation*}
\begin{aligned}
c(t)&= \frac{u_l + u_r}{2}t,\\
a_{\epsilon}(t)&= \Big(\frac{u_l+u^*_{\epsilon}}{2}-\frac{\epsilon}{2}-\frac{1}{2}\sqrt{\epsilon^2+\frac{2\epsilon(\rho_l+\rho^*_{\epsilon})(f(\rho_l)-f(\rho^*_{\epsilon}))}{(\rho_l-\rho^*_{\epsilon})}}\Big)t-c(t) ,\\
b_{\epsilon}(t)&= c(t)-\Big(\frac{u_r+u^*_{\epsilon}}{2}-\frac{\epsilon}{2}+\frac{1}{2}\sqrt{\epsilon^2+\frac{2\epsilon(\rho_r+\rho^*_{\epsilon})(f(\rho_r)-f(\rho^*_{\epsilon}))}{(\rho_r-\rho^*_{\epsilon})}}\Big)t,\\
d_{\epsilon}(t)&=\Big[\frac{u_r- u_l}{2}+\frac{1}{2}\Big(\sqrt{\epsilon^2+\frac{2\epsilon(\rho_l+\rho^*_{\epsilon})(f(\rho_l)-f(\rho^*_{\epsilon}))}{(\rho_l-\rho^*_{\epsilon})}}\\
&+\sqrt{\epsilon^2+\frac{2\epsilon(\rho_r+\rho^*_{\epsilon})(f(\rho_r)-f(\rho^*_{\epsilon}))}{(\rho_r-\rho^*_{\epsilon})}}\Big)\Big] \rho_{\epsilon}^*t.\\
\end{aligned}
\end{equation*}

With the above notations, the formula for $\rho^{\epsilon}$ in equation \eqref{e3.15} can be written in the following form as in the Lemma(3.2):
\begin{equation}
\begin{aligned}
\rho^{\epsilon}=& \rho_l \displaystyle{\chi_{(-\infty, c(t)- a_{\epsilon}(t))}}(x)+\frac{ d_{\epsilon}(t)}{b_{\epsilon}(t)+ a_{\epsilon}(t)}\displaystyle{\chi_{(c(t)-a_{\epsilon}(t), c(t)+b_{\epsilon}(t))}}(x)\\
&+\rho_r\displaystyle{\chi_{(c(t)+b_{\epsilon}(t), \infty)}}(x).
\end{aligned}
\label{e3.17}
\end{equation}
Note that $a_{\epsilon}(t)$ and $b_{\epsilon}(t)$ satisfies the condition of the lemma, i.e, $a_{\epsilon}(t)>0$ and $b_{\epsilon}(t)>0$ for small $\epsilon$.

 Now we are in a position to determine the limit of $d_{\epsilon}(t)$ as $\epsilon \to 0$.
The equation \eqref{e2.11} can also be written in the following  form:

 \begin{equation}
\begin{aligned}
&(\rho_{\epsilon}^{*}+\rho_l)(u^*_{\epsilon}-u_l)=(\rho^*_{\epsilon}-\rho_l)\Big[\epsilon-\sqrt{\epsilon^2+\frac{2\epsilon(\rho^*_{\epsilon}+\rho_l)(f(\rho^*_{\epsilon})-f(\rho_l))}{(\rho^*_{\epsilon}-\rho_l)}}\Big]\\
&(\rho_{\epsilon}^{*}+\rho_r)(u^*_{\epsilon}-u_r)=(\rho^*_{\epsilon}-\rho_r)\Big[\epsilon+\sqrt{\epsilon^2+\frac{2\epsilon(\rho^*_{\epsilon}+\rho_r)(f(\rho^*_{\epsilon})-f(\rho_r))}{(\rho^*_{\epsilon}-\rho_r)}}\Big].
\label{e2.19}
\end{aligned}
\end{equation}
Subtracting second equation from the first in \eqref{e2.19}, we get
\begin{equation*}
\begin{aligned}
&\Big[(u_r-u_l)+\sqrt{\epsilon^2+\frac{2\epsilon(\rho^*_{\epsilon}+\rho_l)(f(\rho^*_{\epsilon})-f(\rho_l))}{(\rho^*_{\epsilon}-\rho_l)}}\\
&+\sqrt{\epsilon^2+\frac{2\epsilon(\rho^*_{\epsilon}+\rho_r)(f(\rho^*_{\epsilon})-f(\rho_r))}{(\rho^*_{\epsilon}-\rho_r)}}\Big]\rho^*_{\epsilon}
+\rho_l(u^*_{\epsilon}-u_l)-\rho_r(u^*_{\epsilon}-u_r)-\epsilon(\rho_r-\rho_l)\\
&=\rho_l\sqrt{\epsilon^2+\frac{2\epsilon(\rho^*_{\epsilon}+\rho_l)(f(\rho^*_{\epsilon})-f(\rho_l))}{(\rho^*_{\epsilon}-\rho_l)}}+\rho_r\sqrt{\epsilon^2+\frac{2\epsilon(\rho^*_{\epsilon}+\rho_r)(f(\rho^*_{\epsilon})-f(\rho_r))}{(\rho^*_{\epsilon}-\rho_r)}}.
\end{aligned}
\end{equation*}
Passing to the limit as $\epsilon \rightarrow  0$,  we get
\begin{equation}
\begin{aligned}
 \displaystyle {\lim_{\epsilon \rightarrow 0}}\Big[(u_r-u_l)&+\sqrt{\epsilon^2+\frac{2\epsilon(\rho^*_{\epsilon}+\rho_l)(f(\rho^*_{\epsilon})-f(\rho_l))}{(\rho^*_{\epsilon}-\rho_l)}}\\
 &+\sqrt{\epsilon^2+\frac{2\epsilon(\rho^*_{\epsilon}+\rho_r)(f(\rho^*_{\epsilon})-f(\rho_r))}{(\rho^*_{\epsilon}-\rho_r)}}\Big]\rho^*_{\epsilon}\\
& =(u_l -u_r)(\rho_l + \rho_r).
\end{aligned}
\label{e2.21}
\end{equation}
This implies 
\begin{equation}
\displaystyle {\lim_{\epsilon \rightarrow 0}} d_{\epsilon}(t)= \frac{1}{2}(u_l -u_r)(\rho_l + \rho_r)t.
\label{e2.22}
\end{equation}
Here in the calculation of \eqref{e2.22}, we have used the fact that  $\displaystyle {\lim_{\epsilon \rightarrow 0}}\,\,\, 2\epsilon\big(f(\rho^*_{\epsilon})-f(\rho_l)\big)=\displaystyle {\lim_{\epsilon \rightarrow 0}}\,\,\, 2\epsilon\big(f(\rho^*_{\epsilon})-f(\rho_r)\big)=l=\frac{1}{4}(u_l-u_r)^2$and $ \displaystyle {\lim_{\epsilon \rightarrow 0}} u^{*}_{\epsilon} =\frac{u_l +u_r}{2}$ from the equation \eqref{e2.15}. The first and the third terms of \eqref{e3.17} converge to  $\rho_l \displaystyle{\chi_{(-\infty, \frac{u_l + u_r}{2}t)}}(x)$  and \\
 $ \rho_r \displaystyle{\chi_{( \frac{u_l + u_r}{2}t, \infty)}}(x)$ respectively.
Hence, employing the above lemma to the second term of \eqref{e3.17}, we get the distribution limit $\rho(x,t)$ as given in the theorem. Note that all the analysis  has been carried out for a subsequence. Since the limit is same for any subsequence, this implies the sequence itself converges to the same limit. This completes the the proof of the theorem.
\end{proof}

Now it remains to show that the limit $(u,\rho)$ found in the theorem above, satisfies the equation \eqref{e1.1}. The limit $(u, \rho)$ satisfies the equation in the sense of Volpert is available in \cite{j3, le1}. There it was shown that $R_t + \overline{u} R_x=0$ , where $\rho= R_x$  and 
$\overline{u}R_x$ is known as Volpert product \cite{v1, d1}. Then $\rho= R_x$  satisfies the equation \eqref{e1.1} in the sense of distribution. The limit $(u, \rho)$ satisfies the equation \eqref{e1.1} is also shown in \cite{sahoo} in the sense of the following definition.
\begin{definition}[\cite{sahoo}]
 Let $u$ is a Borel measurable function and $\rho=d \nu$ is a Radon measure on $\mathbb{R}\times [0, \infty)$. Then $(u, \rho=d\nu)$ is said to be a solution for the system \eqref{e1.1} with initial data \eqref{e1.2} if the following conditions hold.
\begin{equation}
\begin{aligned}
&\int_{\mathbb{R}\times [0, \infty)}( u \phi_t + u \phi_x) dx dt + \int_{\mathbb{R}} u_0 (x)\phi(x, 0)dx=0\\
& \int_{\mathbb{R}\times [0, \infty)}( \phi_t + u \phi_x) d\nu + \int_{\mathbb{R}}\rho_0 (x)\phi(x, 0)dx=0,
\end{aligned}
\label{newdef}
\end{equation}
for any test function $\phi$ supported in $\mathbb{R}\times [0, \infty)$.
\label{defsol}
\end{definition}
Now we state the following theorem and the proof can be found in \cite{sahoo}.
\begin{theorem}[\cite{sahoo}]
For  $u_l > u_r$,  the point wise limit $u$ of $u^{\epsilon}$  and distributional limit of $\rho$ of $\rho^{\epsilon} $ satisfies the equation\eqref{newdef}.
\label{theorem3.5}
\end{theorem}
\section{Entropy and entropy flux pairs}
This section is devoted to construct an explicit entropy-entropy flux pairs for the system \eqref{sc1.6} when $f(\rho)=\frac{\rho^2}{2}$, i.e Brio system. We start with the following definitions\cite{ba1} restricted to $2\times 2$ system, namely
\begin{equation}
\begin{aligned}
u_t +(f_1(u,\rho))_x &=0\\
\rho_t + (f_2 (u, \rho))_x &=0.
\end{aligned}
\label{2by2}
\end{equation}
\begin{definition}
 A continuously differentiable function $\eta:\mathbb{R}^2\mapsto \mathbb{R}$ is called an \textit{entropy} for the system\eqref{2by2} with \textit{entropy flux} $q:\mathbb{R}^2\mapsto \mathbb{R}$ if 
\begin{equation*}
D\eta(u, \rho) . Df(u, \rho)= Dq(u, \rho),
\end{equation*}
where $f(u, \rho)=(f_1(u,\rho), f_2(u,\rho))$. We say $(\eta, q)$  as entropy-entropy flux pair of the system\eqref{2by2}.
\label{entropy}
\end{definition}
\begin{definition}
A weak solution $(u, \rho)$ of the system \eqref{2by2} is called \textit{entropy admissible} if 
\begin{equation*}
\iint_{\mathbb{R}\times (0, \infty)} {\eta(u, \rho)\varphi_t+ q(u, \rho)\varphi_x} \,dx\,dt \geq 0,
\end{equation*}
for every non-negative test function $\varphi: \mathbb{R}\times (0, \infty) \rightarrow \mathbb{R}$ with compact support in $\mathbb{R}\times (0, \infty)$, where $(\eta, q)$ is the entropy-entropy flux pair as in the definition\eqref{entropy}.
\label{entropyinq}
\end{definition}
 Now for the system \eqref{sc1.6}, $f(u, \rho)= \Big(\frac{u^2}{2}+\frac{\epsilon}{2}\rho^2, \,\,\,\, u \rho-\epsilon \rho\Big)$. Therefore $(\eta,q)$ will be an entropy-entropy flux pair of \eqref{sc1.6} if 
\begin{equation*}
\begin{aligned}
\bigg(\frac{\partial \eta}{\partial u}u+\frac{\partial \eta}{\partial \rho}\rho\, ,           \,\, \epsilon\rho \frac{\partial \eta}{\partial u}+ (u-\epsilon)\frac{\partial \eta}{\partial \rho}\bigg)
=\bigg(\frac{\partial q}{\partial u},\frac{\partial q}{\partial \rho}\bigg).
\end{aligned}
\end{equation*}
That is,
\begin{equation}
\begin{aligned}
\frac{\partial q}{\partial u}=\frac{\partial \eta}{\partial u}u+\frac{\partial \eta}{\partial \rho}\rho,\\
\frac{\partial q}{\partial \rho}=\epsilon\rho\frac{\partial \eta}{\partial u}+(u-\epsilon)\frac{\partial \eta}{\partial \rho}.
\label{e3.2}
\end{aligned}
\end{equation}
Eliminating $q$ from \eqref{e3.2}, we have
\begin{equation*}
\epsilon\Big(\rho\frac{\partial^2 \eta}{\partial u^2}-\frac{\partial^2\eta}{\partial \rho\partial\rho}\Big)- \rho\frac{\partial^2 \eta}{\partial \rho^2}=0.
\end{equation*}

One can see that
\begin{equation*}
\eta(u,\rho)=\frac{1}{2}u^2 +\frac {\epsilon}{2}\rho^2
\end{equation*}
 is a solution of above the equation which is a strictly convex (since $D^2\eta > 0$)  and the corresponding entropy flux is 
 \begin{equation*}
 q(u,\rho)=\frac{1}{3}u^3 + \Big(u-\frac{\epsilon}{2}\Big)\epsilon \rho^2.
 \end{equation*}

By constructing an explicit entropy-entropy flux pairs for Brio system, we show here that our solution constructed in the previous section for Riemann type initial data ($u_l > u_r$) which can also be treated as a solution for Brio system if we plug $f(\rho)=\frac{\rho^2}{2}$ into the solution, is entropy admissible in the sense of the above definition\eqref{entropyinq}.

For that we calculate
 \begin{equation}
 \begin{aligned}
 \eta_t+ q_x=& -s_{1,\epsilon}\bigg(\frac{1}{2}u^{*2}_{\epsilon}+\frac{\epsilon}{2} \rho^{*2}_{\epsilon}-\frac{1}{2}u^{2}_{l}-\frac{\epsilon}{2}\rho^2_l\bigg)\delta_{x=s_{1,\epsilon}t}\\
 &-s_{2,\epsilon}\bigg(\frac{1}{2}u^{2}_{r}+\epsilon e^{\rho_r}-\frac{1}{2}u^{*2}_{\epsilon}-\frac{\epsilon}{2}{\rho^*_{\epsilon}}^2\bigg)\delta_{x=s_{2,\epsilon}t} \\ 
 &+\bigg(\frac{1}{3}u^{*3}_{\epsilon}+(u^*_{\epsilon}-\frac{\epsilon}{2})\epsilon \rho^{*2}_{\epsilon}-\frac{1}{3}u^3_l-(u_l-\frac{\epsilon}{2})\epsilon\rho_l^2\bigg)\delta_{x=s_{1,\epsilon}t}\\
 &+\bigg(\frac{1}{3}u^3_r-(u_r-\frac{\epsilon}{2})\rho_r^2-\frac{1}{3}u^{*3}_{\epsilon}-(u^*_{\epsilon}-\frac{\epsilon}{2})\epsilon\rho^{*2}_{\epsilon}\bigg)\delta_{x=s_{2,\epsilon}t},
\label{e3.6}
 \end{aligned}
 \end{equation}
 where $s_{1,\epsilon}$ and $s_{2,\epsilon}$ denote 1-shock speed and 2-shock speed respectively. So from \eqref{shock speed} $s_{1,\epsilon}$ and $s_{2,\epsilon}$ can be written as
 \begin{equation*}
 \begin{aligned}
 s_{1,\epsilon}=\Big(\frac{u_l+u^*_{\epsilon}}{2}-\frac{\epsilon}{2}-\frac{1}{2}\sqrt{\epsilon(\rho_l+\rho^*_{\epsilon})^2+\epsilon^2}\Big),\\
 s_{2,\epsilon}=\Big(\frac{u_r+u^*_{\epsilon}}{2}-\frac{\epsilon}{2}+\frac{1}{2}\sqrt{\epsilon(\rho_r+\rho^*_{\epsilon})^2+\epsilon^2}\Big). 
 \end{aligned}
 \end{equation*}\\
 One can observe  that to show $\eta(u,\rho)$ and $q(u,\rho)$ satisfies the entropy inequality for small $\epsilon$, we must treat the coefficients  $\delta_{x=s_1t}$ and $\delta_{x=s_2t}$ separately. We show that each of the coefficient will be negative as $\epsilon$ tends to zero. let us first consider the coefficient of $\delta_{x=s_1t}$.
  \begin{equation}
\begin{aligned}
& \textnormal{Coefficient of $\delta_{x=s_{1,\epsilon}t}$}\\&=\underbrace{-s_{1}\bigg(\frac{1}{2}u^{*2}_{\epsilon}+\frac{\epsilon}{2} \rho^{*2}_{\epsilon}-\frac{1}{2}u^{2}_{l}-\frac{\epsilon}{2}\rho^2_l\bigg)}_\text{I}+\underbrace{\bigg(\frac{1}{3}u^{*3}_{\epsilon}+(u^*_{\epsilon}-\frac{\epsilon}{2})\epsilon \rho^{*2}_{\epsilon}-\frac{1}{3}u^3_l-(u_l-\frac{\epsilon}{2})\epsilon\rho_l^2\bigg)}_\text{II}
 \label{e4.9}
\end{aligned}
 \end{equation}
 From \eqref{e2.11} we have $(u^*_{\epsilon},\rho^*_{\epsilon})$ satisfies the following equations.
\begin{equation}
\begin{aligned}
u^*_{\epsilon}-u_l=\frac{\rho^*_{\epsilon}-\rho_l}{\rho^*_{\epsilon}+\rho_l}\Big[\epsilon-\sqrt{\epsilon^2+\epsilon(\rho^*_{\epsilon}+\rho_l)^2}\Big],\\
u^*_{\epsilon}-u_r=\frac{\rho^*_{\epsilon}-\rho_r}{\rho^*_{\epsilon}+\rho_r}\Big[\epsilon-\sqrt{\epsilon^2+\epsilon(\rho^*_{\epsilon}+\rho_r)^2}\Big].
\end{aligned}
\end{equation}
 Now similarly as in Theorem $3.3$ we have
 \begin{equation*}
 \begin{split}
 u^*_{\epsilon}-u_l&=-\sqrt{l}\\
 u^*_{\epsilon}-u_r&=\sqrt{l}
 \end{split}
 \end{equation*} where
 \begin{equation}
 \displaystyle {\lim_{\epsilon \rightarrow 0}} \epsilon(\rho^*_{\epsilon}+\rho_l)^2=\displaystyle {\lim_{\epsilon \rightarrow 0}}\epsilon(\rho^*_{\epsilon}+\rho_r)^2=\displaystyle {\lim_{\epsilon \rightarrow 0}}\epsilon \rho^{*2}_{\epsilon}=\frac{1}{4}(u_r-u_l)^2.
 \label{e4.10}
 \end{equation}
 Now using \eqref{e4.10} and observing that $s_1 \rightarrow \frac{(u_l+u_r)}{2}$,  one can see 
 \begin{equation*}
 I\rightarrow \frac{-(u_l+u_r)(u_r^2-u_l^2)}{8}\,\,\, \textnormal{ as $\epsilon \rightarrow 0$}.
 \end{equation*}
 
 Again using \eqref{e4.10},  a simple calculation yields
 \begin{equation*}
 II \rightarrow \frac{(u_r^3-u_l^3)}{6}\,\,\, \textnormal{ as $\epsilon \rightarrow 0$}.
 \end{equation*}
 
 Therefore from the equation\eqref{e4.9}, 
 \begin{equation*}
  \textnormal{Coefficient of $\delta_{x=s_1t}$}= I+II \rightarrow \frac{(u_r-u_l)(u_l-u_r)^2}{24}\,\,\, \textnormal{ as $\epsilon \rightarrow 0$}.
 \label{e4.14}
 \end{equation*}
  Since $u_l>u_r$,  $ \textnormal{Coefficient of $\delta_{x=s_1t}$}= I+II <0$ for small $\epsilon$.
 In a similar way the coefficients of $\delta_{x=s_2 t}$ can be handled.
\begin{remark}
It is well known that if $\eta$ be a smooth entropy of the system \eqref{2by2} with the entropy flux $q$ and if one assumes that the Hessian $D^2\eta > 0$, then for genuinely non-linear characteristic fields the entropy inequality $\eta(u)_t+q(u)_x \leq 0$ is satisfied for sufficiently close initial data. Details can be found in \cite{ba1}. Here it is worth mentioning that our proof is independent of the closeness of the initial data, however it depends on the smallness of $\epsilon$.
\end{remark}

\section{formation of contact discontinuity and cavitation for $u_l\leq u_r$}
In this section we discuss other two cases, i.e, $u_l = u_r$ and  $u_l < u_r$. The discussion in this section is a mere repetition of the steps of \cite{sahoo} except the fact that here we have two different shock speeds. \\
\textbf{Case I $(u_l=u_r)$}:
For $u_l = u_r$, initial data is
\begin{equation*}
\begin{pmatrix}
         u_0 (x)   \\
            \rho_0 (x) \\
         \end{pmatrix}
   =
\begin{cases}
\begin{pmatrix}
         u_l  \\
            \rho_l \\
         \end{pmatrix},\,\,\,\,\,\,\, \textnormal{if}\,\,\,\,\,\,\  x<0\\
 \begin{pmatrix}
         u_l  \\
            \rho_r \\
         \end{pmatrix},\,\,\,\,\,\,\, \textnormal{if}\,\,\,\,\,\,\  x>0.
\end{cases}
\end{equation*}
Now if $\rho_l=\rho_r,$  we have the trivial solution $u(x,t)=u_l$ and $\rho(x,t) = \rho_l$.  Another two possibilities are  $\rho_r < \rho_l$ or  $\rho_r > \rho_l$.\\
\textit{Subcase I($\rho_r < \rho_l$)}: In this case, we start traveling from the state $(u_l,\rho_l)$ in the curve $R_1$  to reach at $( u^*_{\epsilon}, \rho^*_{\epsilon})$, then from $( u^*_{\epsilon}, \rho^*_{\epsilon})$ we travel by $S_2$ to reach at $(u_l,\rho_r)$.
 1-rarefaction curve $R_1$ through $(u_l,\rho_l)$ is obtained solving the differential equation
\begin{equation}
\frac{du}{d\rho}=\frac{\epsilon-\sqrt{4\epsilon \rho f^{\prime}(\rho)+\epsilon^2}}{2\rho},\,\,\,\,\,\,\,\,\,\,\, u(\rho_l)=u_l
\label{5.6}
\end{equation}
Therefore the branch of the curve satisfying \eqref{5.6} can be parameterized by a $C^1$ function $u_1:[\rho_r,\rho_l]\rightarrow[u_l,\infty)$ with parameter $\rho$.
Since $\rho > 0$,  we see that $u_1$ is decreasing. Therefore, $u_1(\rho_r) > u_l$.\\ 
Any state $(u,\rho)$ connected to the end state  $(u_l,\rho_r)$ by admissible 2-shock curve $S_2$ satisfies the following equation:
\begin{equation}
(u-u_l)=\frac{\rho-\rho_r}{\rho+\rho_r}\Big[\epsilon+\sqrt{\epsilon^2+\frac{\epsilon(\rho+\rho_r)(f(\rho)-f(\rho_r))}{(\rho-\rho_r)}}\Big],\,\,\, \rho_r <\rho < \rho_l ; \,\,\,\,u>u_l
\label{5.2}
\end{equation}
and
\begin{equation}
s >\lambda_2 ({u},  {\rho}), \,\,   \lambda_1 (u_l,\rho_r) < s < \lambda_2 (u_l, \rho_r),\,\, \textnormal{where}\,\,\,  s=\frac{\rho u-\rho_r u_l}{\rho-\rho_r}-\epsilon.
\label{5.3}
\end{equation}
Our claim is that for every fixed $\rho >\rho_r$ there exists a unique $u>u_l$ such that the equation \eqref{5.2} holds.
For that let us define
\begin{equation*}
F(u):= u-u_l.
\end{equation*}
Since $F(u_l)=0$  and $F(u) \rightarrow \infty$ as $u \rightarrow \infty$, we have $F([u_l, \infty))= [0, \infty)$. Since $\rho >\rho_r$, right hand side of \eqref{5.2} is positive. Therefore for the given $\rho> \rho_r$, there exists a $u>u_l$ such that 
\begin{equation*}
F(u)=\frac{\rho-\rho_r}{\rho+\rho_r}\Big[\epsilon+\sqrt{\epsilon^2+\frac{2\epsilon(\rho+\rho_r)(f(\rho)-f(\rho_r))}{(\rho-\rho_r)}}\Big].
\end{equation*}
Also observe that $F(u)$ is an increasing function for all $u$ since $F^{\prime}(u)=1$ , $u$ is unique for the given $\rho$.

Similarly in Theorem 3.1, the branch of the curve satisfying \eqref{5.2} and \eqref{5.3} can be parameterized by a $C^1$-function $u_2(\rho)=u_2 :[\rho_r,\rho_l]\rightarrow[u_l,\infty)$ satisfying
\begin{equation}
F(u_2(\rho))=(u_2(\rho)-u_l)=\frac{\rho-\rho_r}{\rho+\rho_r}\Big[\epsilon+\sqrt{\epsilon^2+\frac{2\epsilon(\rho+\rho_r)(f(\rho)-f(\rho_r))}{(\rho-\rho_r)}}\Big]
\label{e5.4}
\end{equation}
Note that $u_2(\rho_r)=u_l$ 
and it is clear from the above equation \eqref{e5.4} that the function $u_2$ is well defined. One can easily check that the function $u_2$ is increasing in the interval $(\rho_r,\rho_l)$.  In fact,  differentiating the above equation \eqref{e5.4} we get,
\begin{equation*}
{u_2}^{\prime}(\rho)=\frac{\epsilon (\rho-\rho_r)\frac{d}{d\rho}\Big[\frac{(\rho+\rho_r)(f(\rho)-f(\rho_r))}{(\rho-\rho_r)}\Big]}{\sqrt{\epsilon^2+\frac{2\epsilon(\rho+\rho_r)(f(\rho)-f(\rho_r))}{(\rho-\rho_r)}}}
+\Big[\epsilon+\sqrt{\epsilon^2+\frac{2\epsilon(\rho+\rho_r)(f(\rho)-f(\rho_r))}{(\rho-\rho_r)}}\Big].
\end{equation*}
Since $\rho>\rho_r$ and $\rho_r>0$, in the view of \eqref{diff3.16} right hand side of above equation is positive for any $\epsilon>0$. That is,
${u_2}^{\prime}(\rho)>0$.

From the above analysis, there exists an intermediate state $\rho^*_{\epsilon} \in (\rho_r, \rho_l)$ such that $u_1(\rho^*_{\epsilon})=u_2(\rho^*_{\epsilon})=u^*_{\epsilon}.$
Hence the solution for \eqref{sc1.6} is given by:
\begin{equation*}
\begin{aligned}
u^{\epsilon}=\left\{
	\begin{array}{llll}
		u_l  & \mbox{if } x <\lambda_1(u_l,\rho_l)t \\
		R^{u}_{1}(x/t)(u_l,\rho_l) & \mbox{if } \lambda_1(u_l,\rho_l)t< x<
		  \lambda_1(u^*_{\epsilon},\rho^*_{\epsilon})t \\
                  u^*_{\epsilon} &\mbox{if} \lambda_1(u^*_{\epsilon},\rho^*_{\epsilon})t< x< s_{2,\epsilon}(u_l,\rho_r,u^*_{\epsilon},\rho^*_{\epsilon})t
                    \\
                     u_r                   &\mbox{if }  x>s_{2,\epsilon}(u_l,\rho_r,u^*_{\epsilon},\rho^*_{\epsilon}) t
	\end{array}
\right.
\end{aligned}
\end{equation*}
and 
\begin{equation*}
\rho^{\epsilon}=\left\{
	\begin{array}{llll}
		\rho_l  & \mbox{if } x <\lambda_1(u_l,\rho_l)t \\
		R^{\rho}_{1}(x/t)(u_l,\rho_l) & \mbox{if } \lambda_1(u_l,\rho_l)t< x<
		  \lambda_1(u^*_{\epsilon},\rho^*_{\epsilon})t \\
                  \rho^*_{\epsilon} &\mbox{if} \lambda_1(u^*_{\epsilon},\rho^*_{\epsilon})t < x< s_2(u_l,\rho_r,u^*_{\epsilon},\rho^*_{\epsilon})t
                    \\
                     \rho_r                   &\mbox{if }  x>s_2(u_l,\rho_r,u^*_{\epsilon},\rho^*_{\epsilon})t
	\end{array}
\right.
\end{equation*}
Where $R_1(\xi)(\bar{u},\bar{\rho})=(R^{u}_{1}(\xi)(\bar{u},\bar{\rho}),R^{\rho}_{1}(\xi)(\bar{u},\bar{\rho}))$ and $R^{u}_{1}(\xi)(\bar{u},\bar{\rho})$ is obtained by solving 
\begin{equation*}
\frac{du}{d\xi}=\frac{\epsilon-\sqrt{4\epsilon \rho f^{\prime}(\rho)+\epsilon^2}}{2\rho},\,\,\,\,\, u(\lambda_1(\bar{u},\bar{\xi}))=\bar{u}.
\end{equation*}
and
$R^{\rho}_{1}(\xi)(\bar{u},\bar{\rho})$ is obtained by solving 
\begin{equation*}
\frac{d\rho}{d\xi}=1, \,\,\,\,\,\,\, \rho(\lambda_1(\bar{u},\bar{\xi}))=\bar{\rho}.
\end{equation*}
and 
\begin{equation*}
s_{2,\epsilon}(u_l,\rho_r,u^*_{\epsilon},\rho^*_{\epsilon})=\frac{u_l+u^*_{\epsilon}}{2}-\frac{\epsilon}{2}+\frac{1}{2}\sqrt{\epsilon^2+\frac{2\epsilon(\rho_r+\rho^*_{\epsilon})(f(\rho_r)-f(\rho^*_{\epsilon}))}{(\rho_r-\rho^*_{\epsilon})}}.
\end{equation*}\\
\textit{Sub-case II ($\rho_l < \rho_r$)}: In a similar way one can start from $(u_l,\rho_l)$ and reach at $(u^*_{\epsilon},\rho^*_{\epsilon})$ by $S_1$ and from $(u^*_{\epsilon},\rho^*_{\epsilon})$ to $ (u_l,\rho_r)$ by $R_2$. Therefore the solution is given by :
\begin{equation*}
u^{\epsilon}=\left\{
	\begin{array}{llll}
		u_l  & \mbox{if } x < s_{1,\epsilon}(u_l,\rho_l, u^*_{\epsilon}, \rho^*_{\epsilon})t \\
		u^*_{\epsilon} & \mbox{if } s_{1,\epsilon}(u_l,\rho_l, u^*_{\epsilon}, \rho^*_{\epsilon})t< x< \lambda_2(u^*_{\epsilon},\rho^*_{\epsilon})t \\
		  
                    R^{u}_{2}(x/t)(u^*_{\epsilon},\rho^*_{\epsilon}) &\mbox{if }  \lambda_2(u^*_{\epsilon},\rho^*_{\epsilon})t < x< \lambda_2(u_l,\rho_r)t
                     \\
                     u_r                   &\mbox{if }  x>\lambda_2(u_l,\rho_r)t
	\end{array}
\right.
\end{equation*}
and

\begin{equation*}
\rho^{\epsilon}=\left\{
	\begin{array}{llll}
		
		\rho_l  & \mbox{if } x < s_{1,\epsilon}(u_l,\rho_l, u^*_{\epsilon}, \rho^*_{\epsilon})t \\
		\rho^*_{\epsilon} & \mbox{if } s_{1,\epsilon}(u_l,\rho_l, u^*_{\epsilon}, \rho^*_{\epsilon})t< x<\lambda_2(u^*_{\epsilon},\rho^*_{\epsilon})t
		\\
                    R^{\rho}_{2}(x/t)(u^*_{\epsilon},\rho^*_{\epsilon}) &\mbox{if } \lambda_2(u^*_{\epsilon},\rho^*_{\epsilon})t < x<\lambda_2(u_l,\rho_r)t
                      \\
                     \rho_r                   &\mbox{if }  x>\lambda_2(u_l,\rho_r)t
	\end{array}
\right.
\end{equation*}
where $R_2(\xi)(\bar{u},\bar{\rho})=(R^{u}_{2}(\xi)(\bar{u},\bar{\rho}),R^{\rho}_{2}(\xi)(\bar{u},\bar{\rho}))$ and $R^{u}_{2}(\xi)(\bar{u},\bar{\rho})$ is obtained by solving 
\begin{equation*}
\frac{du}{d\xi}=\frac{\epsilon+\sqrt{4\epsilon \rho f^{\prime}(\rho)+\epsilon^2}}{2\rho},\,\,\,\,\, u(\lambda_2(\bar{u},\bar{\xi})=\bar{u}.
\end{equation*}
and
$R^{\rho}_{2}(\xi)(\bar{u},\bar{\rho})$ is obtained by solving 
\begin{equation*}
\frac{d\rho}{d\xi}=1, \,\,\,\,\,\,\, \rho(\lambda_2(\bar{u},\bar{\xi})=\bar{\rho}.
\end{equation*}
and
\begin{equation*}
s_{1,\epsilon}(u_l,\rho_l, u^*_{\epsilon}, \rho^*_{\epsilon})=\frac{u_l+u^*_{\epsilon}}{2}-\frac{\epsilon}{2}-\frac{1}{2}\sqrt{\epsilon^2+\frac{2\epsilon(\rho_l+\rho^*_{\epsilon})(f(\rho_l)-f(\rho^*_{\epsilon}))}{(\rho_l-\rho^*_{\epsilon})}}.
\end{equation*}\\
Now our aim is to find the limit of $(u^{\epsilon},\rho^{\epsilon})$ as $\epsilon \rightarrow  0$ in both of the above cases.
Since $\rho^*_{\epsilon} \in (\rho_l,\rho_r)$ or $\rho^*_{\epsilon} \in (\rho_r, \rho_l)$ this implies $\rho^*_{\epsilon}$ is bounded. Also from the above analysis it is evident that $\rho^*_{\epsilon}$ and $u^*_{\epsilon}$ satisfies 1-shock curve and 2-shock curve. This implies
\begin{equation}
\begin{split}
(u^*_{\epsilon}-u_l)=\frac{\rho^*_{\epsilon}-\rho_l}{\rho^*_{\epsilon}+\rho_l}\Big[\epsilon-\sqrt{\epsilon^2+\frac{2\epsilon(\rho^*_{\epsilon}+\rho_l)(f(\rho^*_{\epsilon})-f(\rho_l))}{(\rho^*_{\epsilon}-\rho_l)}}\Big],\,\,\, \rho_r >\rho^*_{\epsilon} >\rho_l ; \,\,\,\,u^*_{\epsilon}<u_l\\
(u^*_{\epsilon}-u_l)=\frac{\rho^*_{\epsilon}-\rho_r}{\rho^*_{\epsilon}+\rho_r}\Big[\epsilon+\sqrt{\epsilon^2+\frac{2\epsilon(\rho^*_{\epsilon}+\rho_r)(f(\rho^*_{\epsilon})-f(\rho_r))}{(\rho^*_{\epsilon}-\rho_r)}}\Big],\,\,\, \rho_r <\rho^*_{\epsilon} < \rho_l ; \,\,\,\,u^*_{\epsilon}>u_l
\end{split}
\label{ss5.5}
\end{equation}
Since right hand side of \eqref{ss5.5} is bounded, as $\epsilon \rightarrow 0$ we get, $\lim_{\epsilon\to 0}u^*_{\epsilon}=u_l$.
Therefore the solution $(u^{\epsilon},\rho^{\epsilon}) \rightarrow (u,\rho)$ as $\epsilon \rightarrow 0$ where $(u,\rho)$ is given by:
\begin{equation*}
\rho=\left\{
	\begin{array}{ll}
		\rho_l  & \mbox{if } x < u_l t \\
		 \rho_r  &\mbox{if }  x>u_l t.
	\end{array}
\right.
\end{equation*}
and
\begin{equation*}
u=\left\{
	\begin{array}{ll}
		u_l  & \mbox{if } x < u_l t \\
		 u_r  &\mbox{if }  x>u_l t.
	\end{array}
\right.
\end{equation*}
Since here $u_l=u_r$ we have $u \equiv u_l$.\\\\
\textbf{Case II $(u_l<u_r)$} : 
It can be observed that solution for this case is exactly same as the solution for the case $u_l<u_r$ described in \cite{sahoo}. For the sake of completeness we include here that part of the result from \cite{sahoo}. The 1st-rarefaction curve passing through $(u_l,\rho_l)$ is given by the solution of the following Cauchy problem:
\begin{equation*}
\frac{du}{d\rho}=\frac{\epsilon-\sqrt{4\epsilon \rho f^{\prime}(\rho)+\epsilon^2}}{2\rho},\,\,\,\,\,\,\,\,\,\,\, u(\rho_l)=u_l,\;\;\;\;\;\; \rho<\rho_l.
\end{equation*}
Note that for this case it does not matter whether $\rho_l<\rho_r$ or $\rho_l>\rho_r$. Therefore without loss of any generality one can take $\rho_l>\rho_r>0$. Now a branch of $R_1$ can be parameterized by a differentiable function $u_1:[0, \rho_l]\rightarrow [u_l,\infty)$ with a parameter $\rho$. Explicitly $u_1$ can be written as
\begin{equation}
u_1(\rho)-u_l=\int_{\rho}^{\rho_l} \frac{\epsilon-\sqrt{4\epsilon \xi f^{\prime}(\xi)+\epsilon^2}}{2\xi} d\xi.
\label{r5.6}
\end{equation}

Since $\rho \in [0,\rho_l]$ is bounded and $\rho>0$, the above integral goes to zero as $\epsilon$ approaches to zero. Therefore we have $u_{1}(\rho) \rightarrow u_l$ as $\epsilon \rightarrow 0$ decreasingly. Similarly, the 2nd-rarefaction curve is given by the solution of then Cauchy problem :
\begin{equation}
\frac{du}{d\rho}=\frac{\epsilon+\sqrt{4\epsilon \rho f^{\prime}(\rho)+\epsilon^2}}{2\rho},\,\,\,\,\,\,\,\,\,\,\,\;\;\;\;\;\; \rho<\rho_r,\,\,\,\,\,\,\,\,\,\, u(\rho_r)=u_r.
\label{5.20}
\end{equation}
Let $u_2:[0,\rho_r]\rightarrow(-\infty, u_r]$ is differentiable and parameterized branch of $R_2$ satisfying \eqref{5.20} and can be written as
\begin{equation*}
 u_2(\rho)-u_r=\int_{\rho}^{\rho_r}\frac{\epsilon+\sqrt{4\epsilon \xi f^{\prime}(\xi)+\epsilon^2}}{2\xi} d\xi.
\end{equation*}
Since $\rho \in [0,\rho_r]$ and $\rho>0$, using the same argument as above, we have $u_2(\rho) \rightarrow u_r$ as $\epsilon \rightarrow 0$ increasingly.
Since $u_l<u_r$, by the above calculation one can see $u_1(0)< u_2(0)$ for small $\epsilon$.
In this case the complete solution is given by:
\begin{equation}
u^{\epsilon}=\left\{
	\begin{array}{lllll}
		u_l  & \mbox{if } x < \lambda_1(u_l,\rho_l)t \\
		R^{u}_{1}(x/t)(u_l,\rho_l) & \mbox{if } \lambda_1(u_l,\rho_l)t< x < \lambda_1(u^{*(1)}_{\epsilon},0)t \\
                    x/t &\mbox{if } \lambda_1(u^{*(1)}_{\epsilon},0)t < x <  \lambda_2(u^{*(2)}_{\epsilon},0)t \\
                     R^{u}_{2}(x/t)(u^{*(2)}_{\epsilon},0) & \mbox{if }  \lambda_2(u^{*(2)}_{\epsilon},0)t <x< \lambda_2(u_r,\rho_r)t  \\
                      u_r                & \mbox{if } x>\lambda_2(u_r,\rho_r)t.
	\end{array}
\right.
\label{e5.22}
\end{equation}
and
\begin{equation}
\rho^{\epsilon}=\left\{
	\begin{array}{lllll}
		\rho_l  & \mbox{if } x < \lambda_1(u_l,\rho_l)t \\
		R^{\rho}_{1}(x/t)(u_l,\rho_l) & \mbox{if } \lambda_1(u_l,\rho_l)t< x < \lambda_1(u^{*(1)}_{\epsilon},0)t \\
                    0 &\mbox{if } \lambda_1(u^{*(1)}_{\epsilon},0)t < x <  \lambda_2(u^{*(2)}_{\epsilon},0)t \\
                     R^{\rho}_{2}(x/t)(u^{*(2)}_{\epsilon},0) & \mbox{if }  \lambda_2(u^{*(2)}_{\epsilon},0)t <x< \lambda_2(u_r,\rho_r)t  \\
                      \rho_r                & \mbox{if } x>\lambda_2(u_r,\rho_r)t,
	\end{array}
\right.
\label{last}
\end{equation}
where $R^u_1(.)$, $R^{\rho}_1(.)$, $R^u_2(.)$, $R^{\rho}_2(.)$ are defined as above.
\\
Now we find the limit of $(u^{\epsilon},\rho^{\epsilon})$ as $\epsilon \rightarrow 0$.
Since $u^{*(1)}_{\epsilon}=u_{1}(0)$, we have  $u^{*(1)}_{\epsilon}\rightarrow u_l$ and similarly $u^{*(2)}_{\epsilon} \rightarrow u_r$ as $\epsilon \rightarrow 0$.
 After passing to the limit in \eqref{e5.22} and \eqref{last} as $\epsilon$ tends to zero, we get
\begin{equation*}
u(x,t)=\left\{
	\begin{array}{llll}
		u_l  & \mbox{if } x < u_lt \\
		 x/t &\mbox{if } u_lt < x <  u_rt \\
                      u_r                & \mbox{if } x> u_rt 
	\end{array}
\right.
\end{equation*}
and
\begin{equation*}
\rho(x,t)=\left\{
	\begin{array}{llll}
		\rho_l  & \mbox{if } x < u_lt \\
		 0 &\mbox{if } u_lt < x <  u_rt \\
                      \rho_r                & \mbox{if } x> u_rt .
	\end{array}
\right.
\end{equation*}

\begin{remark}
In equation \eqref{e5.22}, one has to take $u^{\epsilon}(x,t)= \frac{x}{t}$ in the region $\lambda_1(u^{*(1)}_{\epsilon},0)t < x <  \lambda_2(u^{*(2)}_{\epsilon},0)t$. This kind of choice gives an unique entropy solution. In fact,  since $\rho=0$ in this region, the first equation of \eqref{sc1.6} becomes the well known Burgers equation and $u(x,t)=\frac{x}{t}$ is the unique entropy solution for the rarefaction case of Burgers equation. 
\end{remark}

\section{concluding remarks and further possibilities}
1. Theorem $1.1$ can be achieved by combining the results Theorem $3.1$, Theorem $3.3$ and the discussion in Section $5$. In this article, we studied the generalized Euler system when $f(\rho)$ and $f^{\prime}(\rho)$ both are increasing and $g(\rho)$ is any linear decreasing function. We observed that our analysis is still valid for some particular non linear decreasing $g(\rho)$ and particular $f(\rho)$ with the property stated above. For example, if we take $f(\rho)=\frac{\rho^2}{2}$ and $g(\rho)=-\rho^2$, the shock curves passing through $(u_l,\rho_l)$ are the following.
\begin{equation*}
s_1=\big\{(u,\rho):(u-u_l)=(\rho-\rho_l)\Big[\epsilon-\sqrt{\epsilon^2+\epsilon}\Big],\,\,\, \rho > \rho_l;\,\,\,\,u<u_l\big\},
\end{equation*}
\begin{equation*}
s_2=\big\{(u,\rho):(u-u_l)=(\rho-\rho_l)\Big[\epsilon+\sqrt{\epsilon^2+\epsilon}\Big],\,\,\, \rho < \rho_l;\,\,\,\,u<u_l\big\}
\end{equation*}
For the case $u_l>u_r$, one has the existence of intermediate state in the same way as in Theorem $3.1$, however, in this case calculations are much more simpler than the calculations presented here.  One can show that  $\lim_{\epsilon \to 0}\sqrt{\epsilon}\rho^{*}_{\epsilon}$ exists and following the steps of Theorem$3.3$, distributional limit of $(u_{\epsilon}, \rho_{\epsilon})$ as $\epsilon\to0$ can be determined. Finally, the case $u_l\leq u_r$ can be handled in a similar way as in Section $5$.\\\\
2. One can address a similar question with general $g(\rho)$.  Note that the shock curves passing through $(u_l, \rho_l)$ for any general $f(\rho)$ and $g(\rho)$, can be found in the following manner.
\begin{equation*}
s_1=\Big\{(u,\rho):(u-u_l)=\frac{g(\rho_l)-g(\rho)}{(\rho+\rho_l)}\Big[\epsilon-\sqrt{\epsilon^2+\frac{2\epsilon(\rho^2-\rho_l^2)(f(\rho)-f(\rho_l))}{(g(\rho)-g(\rho_l))^2}}\Big]\Big\},
\end{equation*}
\begin{equation*}
s_2=\big\{(u,\rho):(u-u_l)=\frac{g(\rho_l)-g(\rho)}{(\rho+\rho_l)}\Big[\epsilon+\sqrt{\epsilon^2+\frac{2\epsilon(\rho^2-\rho_l^2)(f(\rho)-f(\rho_l))}{(g(\rho)-g(\rho_l))^2}}\Big]\Big\}
\end{equation*}
Next difficulty is to choose the admissible shock curves satisfying Lax entropy inequality and show the existence of intermediate state as in theorem (3.1). Then one needs to determine the 
proper growth condition on $g$ to find the distributional limit of solutions of  the scaled system.

\end{document}